\documentclass[preprint,12pt]{elsarticle}
\usepackage[T1]{fontenc}
\usepackage[utf8]{inputenc}
\usepackage{mathtools,amssymb,amsthm}
\usepackage{graphicx,float,microtype}
\usepackage{hyperref}
\journal{}
\numberwithin{equation}{section}
\theoremstyle{plain}
\newtheorem{theorem}{Theorem}[section]
\newtheorem{lemma}[theorem]{Lemma}
\newtheorem{proposition}[theorem]{Proposition}
\newtheorem{corollary}[theorem]{Corollary}
\theoremstyle{definition}

\newcommand{\R}{\mathbb R}
\newcommand{\C}{\mathbb C}

\newcommand{\eps}{\varepsilon}
\newcommand{\dd}{\,\mathrm d}
\newcommand{\ii}{\mathrm i}

\DeclareMathOperator{\diag}{diag}
\DeclareMathOperator{\tr}{tr}

\hypersetup{pdftitle={Fully Deterministic Shallow Neural Networks with Localized Activations. Part II: Condition Number Estimates},pdfauthor={Ran Bi and Weibing Deng}}
\begin{document}
\begin{frontmatter}
\title{Fully Deterministic Shallow Neural Networks with Localized Activations. Part II: Condition Number Estimates}
\author[nju]{Ran Bi}
\ead{ranbi@smail.nju.edu.cn}
\author[nju]{Weibing Deng}
\ead{wbdeng@nju.edu.cn}
\affiliation[nju]{organization={School of Mathematics, Nanjing University},
                 city={Nanjing},
                 postcode={210093},
                 country={People's Republic of China}}
\begin{abstract}
	We prove positive definiteness and quantitative condition-number bounds
	for continuous Sobolev Gram matrices of deterministic shallow networks
	with Gaussian and tanh-difference activations. The prescribed feature
	families include the optimal Sobolev approximation constructions of Part I.
	On every fixed nonempty bounded Lipschitz domain in $\R^d$, $d\ge2$,
	and for every fixed integer $m\ge0$, the raw and column-normalized
	Gram matrices in the full $H^m$ inner product have spectral condition
	numbers bounded by $\exp(CN^{2/d}/\log(N+2))$ for sufficiently large $N$,
	where $N$ is the feature count and $C>0$ is independent of $N$.
	The upper bounds hold for midpoint and Gauss-Legendre offsets.
	For midpoint offsets and a suitable angular sequence, matching lower
	bounds on fixed ellipsoids establish sharpness of this exponent over
	the stated domain class for raw Gaussian matrices and both raw and
	normalized tanh-difference matrices.
	The proof combines homogeneous differential operators that isolate
	direction groups with one-dimensional coefficient-recovery estimates.
	Two-dimensional $L^2$ experiments on a square illustrate condition-number
	growth and normalization effects at finite resolutions.
\end{abstract}
\begin{keyword}
Gram matrix \sep condition number \sep deterministic neural network \sep localized ridge function \sep coefficient recovery \sep Sobolev approximation
\MSC[2020] Primary 65F35 \sep Secondary 41A46, 65D15, 68T07
\end{keyword}
\end{frontmatter}
\tableofcontents
\section{Introduction}
In this paper, we study the condition numbers of continuous Sobolev Gram matrices for
deterministic families of localized ridge functions on a fixed nonempty
bounded Lipschitz domain $\Omega\subset\R^d$, with $d\ge2$.
As in Part I, the base offset radius is the domain radius
$R:=\sup_{x\in\Omega}|x|$, where $|x|$ denotes the Euclidean norm.
The features use the Gaussian activation $\sigma_G(t)=e^{-t^2}$ or the
fixed tanh difference $\sigma_D(t)=\tanh(t+1)-\tanh(t-1)$.
Their directions come from recursive spherical cubature, and their offsets
are equispaced midpoints or Gauss-Legendre nodes. We prove linear
independence and quantify the sensitivity of both raw and
column-normalized coefficient representations as the activation scales
decrease.

For fixed hidden parameters, constructing an approximation in the
resulting linear trial space amounts to determining the output
coefficients. Classical
constructive approximation results prescribe such parameters
\cite{Mhaskar1996,Petrushev1998}, and recent results establish optimal
Sobolev rates for prescribed $\mathrm{ReLU}^k$ and smooth sigmoidal features
under their stated activation and smoothness conditions~%
\cite{LiuMaoXu2025Relu,FuWang2026Optimal}.
The same linear structure underlies extreme learning machines
\cite{HuangZhuSiew2006} and feature methods for PDEs
\cite{DongLi2021,ChenChiEYang2022}. In these applications, algebraic
systems determine the output coefficients.
Mao~et~al.\ \cite{MaoXuXu2026Linearized} report severe ill-conditioning
of variational systems for fixed $\mathrm{ReLU}^k$ and ordinary tanh
features.

Part I \cite{Bi2026Approximation} constructs feature families of this form
whose spans attain the optimal Sobolev approximation order. For each fixed
$s>0$, its geometric normalization and parameter conditions yield
explicit hidden parameters that are independent of the target, together
with a bounded linear approximation operator. For all sufficiently large
frequency resolutions $M$, the family contains $N\asymp M^d$ features,
and a single approximant
$v_M[u]$ satisfies
\[
 \|u-v_M[u]\|_{H^m(\Omega)}
 \le C N^{-(s-m)/d}\|u\|_{H^s(\Omega)},
 \qquad m\in\mathbb N_0,\quad m<s,
\]
for every $u\in H^s(\Omega)$, simultaneously for all the indicated
integers $m$. This order is sharp on the Sobolev unit ball among linear
spaces chosen independently of the target. The Fourier construction also
bounds the sum of the absolute values of the output coefficients by a
polynomial in $M$ times $\|u\|_{H^s(\Omega)}$. These coefficients include
the cubature weights.

A coefficient bound for a particular approximation operator does not
control the conditioning of the full feature family. For a fixed integer $m\ge0$
and either activation $\sigma\in\{\sigma_G,\sigma_D\}$, let
$H_{m,\sigma,M}$ denote the continuous Gram matrix of the unweighted
features $g_1,\ldots,g_N$ in the full $H^m(\Omega)$ inner product. Its
smallest eigenvalue satisfies
\[
 \lambda_{\min}(H_{m,\sigma,M})
 =\min_{\|c\|_2=1}
   \left\|\sum_{i=1}^N c_ig_i\right\|_{H^m(\Omega)}^2.
\]
This minimum can be small because individual columns have small norms or
because linear combinations nearly cancel.
For a linearly independent feature family, the condition number
of the synthesis map from Euclidean coefficients to $H^m(\Omega)$ is the
square root of the Gram condition number. The spectral estimates below
therefore quantify coefficient recovery for arbitrary linear combinations
of the prescribed features.

Continuous neural Gram matrices have been studied for one-dimensional
ReLU trial spaces \cite{HongSiegelTanXu2022SpectralBias,CaiDoktorovaFalgoutHerrera2025Ritz}
and in connection with high-frequency approximation
\cite{ZhangZhaoZhongZhou2025HighFrequencies}.
On spheres, Liu~et~al.\ \cite{LiuMaoXu2025Conditioning} obtain sharp
geometric spectral and condition-number estimates for
$\mathrm{ReLU}^k$ features. For Gaussian translates, Erd\'elyi established
Bernstein-type derivative
bounds and Remez inequalities controlling function norms
\cite{Erdelyi2006Gaussian,Erdelyi2009GaussianRemez}.
The analytic localized activations considered here have scales that
decrease with resolution. Their analysis requires coefficient-recovery
estimates at these scales and a means of separating the multivariate
direction groups.

For sampled feature matrices, Chen~et~al.\
\cite{ChenSchaefferWard2022} prove singular-value concentration for
random Fourier features under dimension, separation, and sampling
conditions. Chen and Schaeffer \cite{ChenSchaeffer2024} establish
condition-number estimates for these matrices. Quantitative lower
spectral estimates also hold for nonlinear
random features under specified initialization and data assumptions
\cite{NguyenMondelliMontufar2021,WangZhu2024}.
Under suitable nondegeneracy and sampling assumptions, the evaluation
matrices of analytic random-feature models have maximal rank almost surely
\cite{Holzmueller2021DoubleDescent}. Maximal rank alone gives no
quantitative bound for the smallest nonzero singular value.
For RFM collocation, Ming and Yu
\cite{MingYu2025OneDimension,MingYu2026Spectral} derive rapid singular-value
decay and condition-number lower bounds for Fourier and ordinary tanh
residual matrices with bounded frequency windows. These sampled-matrix
estimates depend on the sampling geometry or residual operator.

To compute output coefficients, ELM methods use Lanczos and SVD
solvers, together with Tikhonov regularization
\cite{TangHan2009Lanczos,MicheEtAl2011TROP}, while high-precision RFM studies
examine solver sensitivity \cite{ChenESun2024}.
Recent RFM methods use preconditioning and feature filtering
\cite{TanChen2026,VanBeekDoleanMoseley2026}, or operator-aware sampling with
residual whitening \cite{LinghuWang2026ResidualChristoffel}.
These methods act on a specified discrete system or modify its
representation. Regularized frame approximation can remain accurate with
an ill-conditioned coefficient representation \cite{AdcockHuybrechs2019}.
Continuous Gram estimates measure coefficient sensitivity in a prescribed
feature family. The stability of a PDE discretization also depends on
the operator, sampling, quadrature, regularization, and solution algorithm.

Our coefficient-recovery argument starts from the feature definitions and
applies to arbitrary coefficient vectors. It does not require the ridge
representation or approximation theorem of Part I. The upper bounds hold
for every fixed integer $m\ge0$ and arbitrary fixed positive offset-margin
constants. This parameter range includes the feature families used for
optimal approximation in Part I. We give the full proofs for the rounded
midpoint grid. A polynomial lower bound on Gauss-Legendre offset spacing
then allows the same recovery arguments to establish the corresponding
upper bound.

The Gaussian scale is $\eps_{G,M}=\sqrt{\log(M+2)}/M$, and the
tanh-difference scale is $\eps_{D,M}=\log(M+2)/M$. Both feature families have
$N\asymp M^d$ features. Here $\kappa_2$ denotes the condition number in the
matrix norm induced by the Euclidean vector norm. We obtain the following
results.
\begin{enumerate}
\item For each fixed integer $m\ge0$, the raw and column-normalized Gram
matrices formed in the full $H^m$ inner product are positive definite on
every fixed nonempty bounded Lipschitz domain for all sufficiently large
$M$, and their spectral condition numbers satisfy
\[
 \kappa_2\le\exp(C_m M^2/\log(M+2)).
\]
The underlying recovery estimate gives a quantitative synthesis lower
bound for both midpoint and Gauss-Legendre offsets. Since $N\asymp M^d$,
the exponent has order $N^{2/d}/\log(N+2)$.
\item The recursive angular grid admits homogeneous selectors of degree
$O_d(M)$ and coefficient norm $\exp(O_d(M\log M))$. These selectors
separate $O(M^{d-1})$ ridge directions without requiring a product over
all directions. Combined with one-dimensional recovery estimates for
each activation, they control the full synthesis map by observation on
an interior ball.
\item For midpoint feature families, fixed ellipsoids and an admissible
angular degree sequence yield matching lower bounds for raw Gaussian
matrices and for raw and normalized
tanh-difference matrices. Gaussian tails produce a large disparity
between column norms. Cancellation of successive tanh tail moments makes
the same exponential order persist after column normalization.
For normalized Gaussian midpoint matrices, we also prove a polynomial
lower bound on every fixed domain and an $\exp(c_mM)$ lower bound on
the same short-axis ellipsoid. These bounds do not match the general
upper estimate.
\end{enumerate}
The matching lower bounds establish the optimal exponential order over the
stated class of fixed domains and admissible angular degree sequences.
The constants may depend on the domain, whose radius determines the base
offset range. Optimality over this class does not imply a matching lower
bound on every fixed domain. The square-domain $L^2$ experiments
illustrate growth and the effect of normalization at finite resolutions.

Section~\ref{gub:sec:actual-dictionary} defines the feature families and states the
main theorem. Section~\ref{gub:sec:recursive-directions} constructs the
direction selectors. The next two sections prove Gaussian and
tanh-difference coefficient recovery. The resulting estimates give the
raw and normalized spectral bounds. We prove the lower bounds in
Section~\ref{gub:sh:section}, and report the square-domain experiments in
Section~\ref{sec:numerics}.

\section{Ridge families and the main theorem}
\label{gub:sec:actual-dictionary}

\subsection{Midpoint ridge families}
We define the ridge feature families and their Gram matrices. The hidden
parameters are prescribed by the deterministic construction of Part I
\cite{Bi2026Approximation}.
Let $d\ge2$ be an integer and let $\Omega\subset\R^d$ be a
nonempty bounded Lipschitz domain. As in Part I, define
\[
 R:=\sup_{x\in\Omega}|x|>0,
\]
where $|x|$ is the Euclidean norm. Equivalently,
$R=\max_{x\in\overline\Omega}|x|$.
Write $B_r(y)$ for the open Euclidean ball of radius $r$ centered at
$y$. Then $\Omega\subset B_R(0)$, and the radius $R$ determines the
base offset range.

We use the two fixed activations
\begin{equation}\label{gub:eq:activations}
 \sigma_G(t)=e^{-t^2},\qquad
 \sigma_D(t)=\tanh(t+1)-\tanh(t-1).
\end{equation}
Both are even and strictly positive on the real axis. For either
$\sigma\in\{\sigma_G,\sigma_D\}$, a ridge feature has the form
\[
 \sigma\!\left(\frac{\omega\cdot x-b}{\eps}\right),
 \qquad x\in\Omega,\quad\omega\in\mathbb S^{d-1},\quad b\ge0,
 \quad\eps>0.
\]
Here $\omega$ is the unit normal to the hyperplane $\omega\cdot x=b$,
$b$ is its geometric offset, and $\eps$ controls the localization width
in the normal direction. The hidden weight vector and additive bias in
this representation are $\omega/\eps$ and $-b/\eps$, respectively.
A common scale and finite sets of directions and offsets determine the
feature family.

Let $M>0$ denote the real frequency resolution parameter. For each $M$,
all features associated with a given activation have the same scale
$\eps_{\sigma,M}$. As in Part I, we take
\begin{equation}\label{gub:eq:scales}
 T_M=\log(M+2),\qquad
 \eps_{G,M}=\frac{\sqrt{T_M}}{M},\qquad
 \eps_{D,M}=\frac{T_M}{M}.
\end{equation}
Since $|\omega\cdot x|<R$ for $x\in\Omega$, we choose positive offsets
in an interval $(0,B_{\sigma,M})$ extending beyond $R$ by the offset
margin $\Delta_{\sigma,M}$. In Part I, this margin controls the error
caused by truncating the offset integral. Fix $a_G,a_D>0$ and set
\begin{equation}\label{gub:eq:margins}
 \Delta_{G,M}=a_G\frac{T_M}{M},\qquad
 \Delta_{D,M}=a_D\frac{T_M^2}{M},\qquad
 B_{\sigma,M}=R+\Delta_{\sigma,M}.
\end{equation}
The subscript $\sigma$ distinguishes the parameters for $\sigma_G$ and
$\sigma_D$. For a fixed activation, write $\eps=\eps_{\sigma,M}$ and
$B_M=B_{\sigma,M}$. For all sufficiently large $M$,
\begin{equation}\label{gub:eq:bounded-envelope}
 R<B_M\le R+1,\qquad 0<\eps\le1.
\end{equation}

For the direction set, fix an integer-valued angular degree
$q_M\asymp M$, with comparison constants independent of $M$. Define
the number of latitude nodes by
\begin{equation}\label{gub:eq:degree-and-step-constant}
 k_M^{\mathrm{lat}}=\left\lceil\frac{q_M+1}{2}\right\rceil.
\end{equation}
Begin with the $q_M+1$ equally spaced points on the unit circle, at angles
$2\pi a/(q_M+1)$ for $0\le a<q_M+1$. For each $p=3,\ldots,d$, lift
each previously constructed point $\eta\in\mathbb S^{p-2}$ to
$(\sqrt{1-t^2}\eta,t)$, where $t$ runs over the zeros of the
degree-$k_M^{\mathrm{lat}}$ Jacobi polynomial with both parameters equal
to $(p-3)/2$. Add the antipodal points and retain each distinct point
once. Section~\ref{gub:sec:recursive-directions} gives the recursive formulas.

Denote the resulting direction set by $X_M$, and let $n_{\theta,M}$ be
the number of circle points after antipodal completion:
\[
 n_{\theta,M}=\begin{cases}
 q_M+1,&q_M+1\text{ even},\\
 2(q_M+1),&q_M+1\text{ odd},
 \end{cases}.
\]
The cardinality of $X_M$ is
\begin{equation}\label{gub:eq:full-direction-count}
 J_M^{\mathrm{eq}}=n_{\theta,M}(k_M^{\mathrm{lat}})^{d-2}
 \asymp M^{d-1}.
\end{equation}
Enumerate $X_M=\{\omega_\nu:1\le\nu\le J_M^{\mathrm{eq}}\}$ in
antipodal pairs, so that
$\omega_{\nu+J_M^{\mathrm{eq}}/2}=-\omega_\nu$ for
$1\le\nu\le J_M^{\mathrm{eq}}/2$.

Fix $\kappa_0>0$ to prescribe the nominal offset spacing $\kappa_0/M$.
For sufficiently large $M$, define the positive offsets by the midpoints
\begin{equation}\label{gub:eq:half-grid}
 \begin{aligned}
 K_M^{\mathrm{eq}}&=\left\lfloor\frac{B_MM}{\kappa_0}\right\rfloor,
 \qquad h_M=\frac{B_M}{K_M^{\mathrm{eq}}},\\
 b_{\ell,M}&=(\ell+\tfrac12)h_M,
 \qquad0\le\ell<K_M^{\mathrm{eq}}.
 \end{aligned}
\end{equation}
The prescribed features and their total number are
\begin{equation}\label{gub:eq:actual-columns}
 \begin{aligned}
 g_{\nu\ell,M}(x)
 &=\sigma\!\left(\frac{\omega_\nu\cdot x-b_{\ell,M}}{\eps_{\sigma,M}}\right),\\
 N_M^{\mathrm{eq}}&=J_M^{\mathrm{eq}}K_M^{\mathrm{eq}}.
 \end{aligned}
\end{equation}
Here $1\le\nu\le J_M^{\mathrm{eq}}$ and $0\le\ell<K_M^{\mathrm{eq}}$.
Each tanh-difference feature contributes one Gram column and is realized
by two ordinary $\tanh$ neurons.
The spherical cubature and offset quadrature weights of Part I are
included in its output coefficients. 
For each activation with the midpoint rule, let
\[
 V_M=\operatorname{span}\{g_{\nu\ell,M}:\omega_\nu\in X_M,
                         \ 0\le\ell<K_M^{\mathrm{eq}}\}.
\]
The span is taken over $\C$. For each $M$, the hidden parameters are
fixed, and $V_M$ is a finite-dimensional linear trial space with
$\dim V_M\le N_M^{\mathrm{eq}}$. Theorem~\ref{gub:thm:main} proves that
the prescribed features are linearly independent and hence form a basis
of $V_M$.

For comparison with the approximation theorem of Part I, assume additionally that
$R\ge1$, as in its geometric normalization, and fix an approximation
regularity order $s>0$.
If $q_M$ and $\kappa_0$ satisfy the sufficient parameter conditions
of Part I and the margin constants also satisfy
\begin{equation}\label{gub:eq:margin-constants}
 a_G^2\ge\max\{s-d/2,1\},\qquad
 2a_D>\max\{s-d/2,\pi+1/2\},
\end{equation}
then, for every $u\in H^s(\Omega)$, every integer $0\le m<s$, and all
sufficiently large $M$,
\[
 \inf_{v\in V_M}\|u-v\|_{H^m(\Omega)}
 \le C(N_M^{\mathrm{eq}})^{-(s-m)/d}
 \|u\|_{H^s(\Omega)},
\]
with $C$ independent of $u$ and $M$. These additional approximation
conditions are not required for the conditioning theorem below.

In dimension two, index the circle directions by
$a=0,\ldots,n_{\theta,M}-1$ and set
$\theta_a=2\pi a/n_{\theta,M}$. The features then have the explicit form
\begin{equation}\label{gub:eq:planar-columns}
 g_{a\ell,M}(x_1,x_2)
 =\sigma\!\left(\frac{x_1\cos\theta_a+x_2\sin\theta_a
             -(\ell+\tfrac12)h_M}{\eps_{\sigma,M}}\right),
\end{equation}
where $0\le\ell<K_M^{\mathrm{eq}}$.
Thus $N_M^{\mathrm{eq}}=n_{\theta,M}K_M^{\mathrm{eq}}$ in the planar
case.

\begin{lemma}\label{gub:lem:rounded-grid}
For all sufficiently large $M$,
\begin{equation}\label{gub:eq:step-bounds}
 \begin{gathered}
 \frac{\kappa_0}{M}\le h_M\le\frac{2\kappa_0}{M},\\
 0\le h_M-\frac{\kappa_0}{M}
 <\frac{\kappa_0}{M K_M^{\mathrm{eq}}}=O(M^{-2}),
 \end{gathered}
\end{equation}
and
\begin{equation}\label{gub:eq:count}
 \frac{R}{2\kappa_0}M\le K_M^{\mathrm{eq}}
 \le\frac{R+1}{\kappa_0}M,\qquad N_M^{\mathrm{eq}}\asymp M^d.
\end{equation}
\end{lemma}
\begin{proof}
Write $B_MM/\kappa_0=K_M^{\mathrm{eq}}+\vartheta_M$, where
$0\le\vartheta_M<1$. For sufficiently large $M$,
\[
 \frac{B_MM}{2\kappa_0}\le K_M^{\mathrm{eq}}\le\frac{B_MM}{\kappa_0}.
\]
Combining this inequality with $R<B_M\le R+1$ gives the offset-count bounds in
\eqref{gub:eq:count}. The exact identity
\begin{equation}\label{gub:eq:rounding-identity}
 h_M=\frac{\kappa_0}{M}
       \left(1+\frac{\vartheta_M}{K_M^{\mathrm{eq}}}\right)
\end{equation}
yields the spacing bounds and their $O(M^{-2})$ difference.
Combining the offset count with \eqref{gub:eq:full-direction-count} gives
$N_M^{\mathrm{eq}}\asymp M^d$.
\end{proof}

For coefficient recovery, take the first $J_M^{\mathrm{eq}}/2$ directions
as representatives of the antipodal pairs. Since the activations are even,
$\sigma((-\omega\cdot x-b)/\eps)=\sigma((\omega\cdot x+b)/\eps)$.
Thus a feature with direction $-\omega$ and positive offset $b$ is
identical to the feature with direction $\omega$ and offset $-b$.
We may therefore reindex the features in \eqref{gub:eq:actual-columns}
by antipodal representatives and signed offsets:
\begin{equation}\label{gub:eq:local-notation}
 J=J_M^{\mathrm{eq}}/2,\quad K=2K_M^{\mathrm{eq}},\quad
 N=N_M^{\mathrm{eq}}=JK.
\end{equation}
\begin{equation}\label{gub:eq:signed-grid}
 \begin{aligned}
 b_j&=\left(j-\frac{K+1}{2}\right)h_M,\qquad1\le j\le K,\\
 g_{\nu j}&=\begin{cases}
 g_{\nu+J,K/2-j,M},&1\le j\le K/2,\\
 g_{\nu,j-K/2-1,M},&K/2<j\le K.
 \end{cases}
 \end{aligned}
\end{equation}
Here $1\le\nu\le J$, and the index $M$ is suppressed in $g_{\nu j}$.
Each $g_{\nu j}$ has representative direction $\omega_\nu$ and
signed offset $b_j$.
The signed grid consists of the positive midpoint offsets and their
negatives. Its central pair is $-h_M/2,h_M/2$, the spacing across zero
is exactly $h_M$, and there is no zero node.
For a coefficient vector $c=(c_{\nu j})\in\C^N$, indexed by
$1\le\nu\le J$ and $1\le j\le K$, define
\begin{equation}\label{gub:eq:dictionary}
 V_c(x)=\sum_{\nu=1}^J\sum_{j=1}^Kc_{\nu j}g_{\nu j}(x).
\end{equation}
For coefficient recovery on $\Omega$, choose a fixed point $x_0\in\Omega$
and radius $\rho>0$ so that
\begin{equation}\label{gub:eq:domain-sandwich}
 \overline{B_\rho(x_0)}\subset\Omega\subset B_R(0).
\end{equation}
We use the fixed interior ball $B_\rho(x_0)$ as an observation region.
Centering the coordinates at $x_0$ gives the offsets
$\beta_{\nu j}=b_j-\omega_\nu\cdot x_0$. With
$C_K=2(R+1)/\kappa_0$ and $B_0=R+1+|x_0|$, they satisfy, for all
sufficiently large $M$ and uniformly in $\nu$,
\begin{equation}\label{gub:eq:shifted-bias-input}
 \begin{aligned}
 K&\le C_KM,\qquad |\beta_{\nu j}|\le B_0,\\
 |\beta_{\nu i}-\beta_{\nu j}|&\ge\kappa_0/M
 \qquad(i\ne j).
 \end{aligned}
\end{equation}
These bounds follow from Lemma~\ref{gub:lem:rounded-grid} and
$|b_j-\omega_\nu\cdot x_0|\le B_M+|x_0|$. Centering preserves the
difference between two offsets in the same direction group.

Unless stated otherwise, positive constants are independent of $M$,
the node indices, and the coefficient vector. They may depend on the
fixed domain, dimension, Sobolev order, construction parameters, and
comparison constants in $q_M\asymp M$. Dimension subscripts in $C_d$
and $O_d$ also suppress these fixed comparison constants.

\subsection{Gauss-Legendre offsets}
\label{gub:sec:gl-offsets}
For the Gauss-Legendre offset rule of Part I, retain the same activation,
direction set $X_M$, scale $\eps$, and interval $(0,B_M)$.
At resolution $M$, let
$K_+^{\mathrm{GL}}=\lceil(q_M+1)/2\rceil$ be the number of positive
offsets. Let $-1<t_1<\cdots<t_{K_+^{\mathrm{GL}}}<1$ be the zeros of
the Legendre polynomial of degree $K_+^{\mathrm{GL}}$. The affine map to
$(0,B_M)$ gives
\begin{equation}\label{gub:eq:gl-biases}
 b_j^{\mathrm{GL}}=\frac{B_M}{2}(1+t_j),
 \qquad 1\le j\le K_+^{\mathrm{GL}}.
\end{equation}
The Gauss-Legendre features are obtained from
\eqref{gub:eq:actual-columns} by replacing $b_{\ell,M}$ with
$b_j^{\mathrm{GL}}$ and using the offset index
$1\le j\le K_+^{\mathrm{GL}}$.
As with the midpoint rule, the quadrature weights are included in the
approximation coefficients. The Gram columns are the unweighted features,
whose total number is
$N_M^{\mathrm{GL}}=J_M^{\mathrm{eq}}K_+^{\mathrm{GL}}\asymp M^d$.
By evenness, the same family can be indexed by the $J$ antipodal
representatives and the signed offsets $\{\pm b_j^{\mathrm{GL}}\}$,
with $K^{\mathrm{GL}}=2K_+^{\mathrm{GL}}$ offsets per representative
direction.

Since $B_M\to R>0$ and $K_+^{\mathrm{GL}}\asymp M$, the classical
endpoint and spacing bounds for Legendre zeros
\cite[Theorem~6.21.2]{Szego1975}, followed by the affine map in
\eqref{gub:eq:gl-biases}, give, for all sufficiently large $M$,
\begin{equation}\label{gub:eq:gl-separation}
 \min_j b_j^{\mathrm{GL}}\ge C^{-1}M^{-2},\qquad
 \min_{i\ne j}|b_i^{\mathrm{GL}}-b_j^{\mathrm{GL}}|\ge C^{-1}M^{-2}.
\end{equation}
The central signed pair is $-b_1^{\mathrm{GL}},b_1^{\mathrm{GL}}$,
with gap $2b_1^{\mathrm{GL}}$. Thus \eqref{gub:eq:gl-separation}
gives a lower bound $C^{-1}M^{-2}$ for the distance between any two
distinct signed offsets.
Within each direction group, centering at $x_0$ preserves these
differences. Every centered offset satisfies
$|\pm b_j^{\mathrm{GL}}-\omega_\nu\cdot x_0|
\le B_M+|x_0|\le B_0$ for $1\le\nu\le J$.
Thus the Gauss-Legendre rule has $K^{\mathrm{GL}}\le CM$,
centered offsets bounded by $B_0$, and separation at least
$C^{-1}M^{-2}$, with $C$ independent of $M$ and the node indices.

\subsection{Sobolev Gram matrices}
Fix a nonnegative integer $m$. For a multi-index
$\alpha=(\alpha_1,\ldots,\alpha_d)$ of nonnegative integers, write
$|\alpha|=\alpha_1+\cdots+\alpha_d$ and let $\partial^\alpha$ denote
the corresponding weak derivative. For $f,g\in H^m(\Omega)$, define
the full Sobolev inner product and norm by
\begin{equation}\label{gub:eq:sobolev}
 \begin{aligned}
 \langle f,g\rangle_m
 &=\sum_{|\alpha|\le m}\int_\Omega
 \partial^\alpha f\,\overline{\partial^\alpha g}\dd x,\\
 \|f\|_m^2&=\langle f,f\rangle_m.
 \end{aligned}
\end{equation}
The inner product is linear in its first argument. Enumerate the
unweighted midpoint features $g_{\nu j}$ in
\eqref{gub:eq:signed-grid} as $g_1,\ldots,g_N$, and write
$c=(c_i)_{i=1}^N$ for the coefficients in the same ordering.
Each feature is strictly positive on $\Omega$, so $\|g_i\|_m>0$.

\begin{samepage}
Define the continuous Gram matrix and its diagonal normalization by
\begin{equation}\label{gub:eq:gram-definition}
 \begin{aligned}
 (H_{m,\sigma,M})_{ij}&=\langle g_j,g_i\rangle_m,
 \qquad 1\le i,j\le N,\\
 D_m&=\diag\bigl(\|g_1\|_m,\ldots,\|g_N\|_m\bigr),\\
 \widetilde H_{m,\sigma,M}&=D_m^{-1}H_{m,\sigma,M}D_m^{-1}.
 \end{aligned}
\end{equation}
\end{samepage}
Then $c^*H_{m,\sigma,M}c=\|V_c\|_m^2$ for every $c\in\C^N$.
The matrix $\widetilde H_{m,\sigma,M}$ is the Gram matrix of the features
$g_i/\|g_i\|_m$, each of which has unit $H^m(\Omega)$ norm.
For the Gauss-Legendre feature family of Section~\ref{gub:sec:gl-offsets},
apply the same definitions to its $N_M^{\mathrm{GL}}$ features and denote
the resulting matrices by $H_{m,\sigma,M}^{\mathrm{GL}}$ and
$\widetilde H_{m,\sigma,M}^{\mathrm{GL}}$.

For an invertible matrix $A$, we use the spectral condition number
$\kappa_2(A)=\|A\|_2\|A^{-1}\|_2$, where $\|\cdot\|_2$ is the matrix
norm induced by the Euclidean vector norm. For a Hermitian positive
definite matrix, $\kappa_2(A)=\lambda_{\max}(A)/\lambda_{\min}(A)$.

\begin{theorem}[Condition-number upper bounds]\label{gub:thm:main}
Consider the midpoint feature family \eqref{gub:eq:actual-columns} with
a fixed domain and dimension, fixed positive parameters
$\kappa_0,a_G,a_D$, and an angular degree sequence $q_M\asymp M$
with fixed comparison constants. For each fixed integer $m\ge0$, there
are constants $C_m,M_0>0$ such that, for each
$\sigma\in\{\sigma_G,\sigma_D\}$ and every real $M\ge M_0$, the Gram
matrix and its normalized form are both positive definite and satisfy
\begin{equation}\label{gub:eq:main-g}
 \max\bigl\{\kappa_2(H_{m,\sigma,M}),\kappa_2(\widetilde H_{m,\sigma,M})\bigr\}
 \le\exp\!\left(C_m\frac{M^2}{\log(M+2)}\right).
\end{equation}
In terms of the feature count $N_M^{\mathrm{eq}}$ for the chosen activation,
\begin{equation}\label{gub:eq:main-N}
 \max\bigl\{\kappa_2(H_{m,\sigma,M}),\kappa_2(\widetilde H_{m,\sigma,M})\bigr\}
 \le\exp\!\left(C_m\frac{(N_M^{\mathrm{eq}})^{2/d}}
 {\log(N_M^{\mathrm{eq}}+2)}\right).
\end{equation}
The constants may depend on $\Omega$, $d$, $m$, the fixed positive
construction parameters, and the comparison constants in $q_M\asymp M$.
They are independent of $M$, the node indices, and the coefficient vector.
\end{theorem}

Corollary~\ref{gub:cor:gl-upper} gives the same condition-number upper
bound for the Gauss-Legendre feature family.

\subsection{Proof outline}
\label{gub:sec:proof-strategy}
The proof of Theorem~\ref{gub:thm:main} uses the coefficient-recovery
inequality
\begin{equation}\label{gub:eq:recovery-goal}
 \|V_c\|_{L^2(\Omega)}
 \ge e^{-CM^2/T_M}\|c\|_2\qquad(c\in\C^N)
\end{equation}
for each activation and all sufficiently large $M$.
Since the full Sobolev norm contains the $L^2$ term, squaring
\eqref{gub:eq:recovery-goal} bounds the smallest eigenvalue of
$H_{m,\sigma,M}$ from below and proves linear independence.
Bounds on the individual column norms control the largest eigenvalue and
the effect of diagonal normalization.

The proof proceeds in four steps.
\begin{enumerate}
\item Section~\ref{gub:sec:recursive-directions} constructs homogeneous
direction selectors of common degree $O_d(M)$ for the
$O(M^{d-1})$ representative directions. Each selector equals one at the
selected representative and zero at the others. The corresponding
differential operator therefore isolates a single direction group.
\item Sections~\ref{gub:g:section} and~\ref{gub:d2:section} establish
one-dimensional recovery estimates for derivatives of the Gaussian and
tanh-difference ridge sums, respectively. The bounds retain explicit
dependence on derivative order and apply at the selector degree, which
grows with $M$. The Sobolev order $m$ remains fixed.
\item Compactly supported test functions within the fixed interior ball
in \eqref{gub:eq:domain-sandwich} transfer one-dimensional recovery to
the full ridge sum. Integration by parts moves the selector derivatives
onto the test functions. Choosing a direction group with maximal
Euclidean coefficient norm yields \eqref{gub:eq:recovery-goal}.
\item Section~\ref{gub:sec:spectrum} combines recovery and feature-norm
bounds to estimate the raw and normalized Gram spectra. The additional
terms in the exponent are of lower order and can be absorbed into the
stated bound. Section~\ref{gub:sec:gl-extension} applies
\eqref{gub:eq:gl-separation} to extend the result to Gauss-Legendre offsets.
\end{enumerate}

\section{Homogeneous direction selectors}
\label{gub:sec:recursive-directions}

Fix a sufficiently large real resolution $M$ and the direction
representatives $\omega_1,\ldots,\omega_J$ from
Section~\ref{gub:sec:actual-dictionary}. Each representative corresponds to a
group of $K$ ridge features with the prescribed signed offsets.
The constants below are independent of $M$ and the node indices.

We construct a differential operator of order $O_d(M)$ for each of the
$J=O(M^{d-1})$ direction groups. Each level of the recursive latitude
grid increases the interpolation degree by $O(M)$, and the number of
levels is fixed. We identify this grid with the prescribed antipodal
completion and construct cardinal polynomials on it. Symmetrization
combines each antipodal pair, and homogenization produces the required
differential operators without changing the polynomial values on the sphere.

Let $r\ge0$ be an integer and $F\in C^r(\R)$. For a homogeneous polynomial
$P(\xi)=\sum_{|\alpha|=r}p_\alpha\xi^\alpha$ on $\R^d$, define
$P(\nabla_x)=\sum_{|\alpha|=r}p_\alpha\partial_x^\alpha$.
For $\omega\in\mathbb S^{d-1}$, the chain rule gives
\begin{equation}\label{dir:operator}
 P(\nabla_x)F(\omega\cdot(x-x_0))
 =P(\omega)F^{(r)}(\omega\cdot(x-x_0)).
\end{equation}
Indeed, for each multi-index $\alpha$ of total degree $r$, the chain rule
gives $\partial^\alpha F(\omega\cdot(x-x_0))
=\omega^\alpha F^{(r)}(\omega\cdot(x-x_0))$. Summing these identities
with coefficients $p_\alpha$ proves \eqref{dir:operator}.
For a fixed target index $1\le\nu\le J$, the cardinal conditions
\[
 P(\omega_\mu)=\delta_{\nu\mu}\qquad(1\le\mu\le J)
\]
ensure that \eqref{dir:operator} isolates the $\nu$th direction group.
Here $\delta_{\nu\mu}$ is the Kronecker symbol.
The monomials must have a common degree, since monomials of different
degrees differentiate $F$ to different orders.

For a polynomial $P=\sum_\alpha p_\alpha x^\alpha$ in any finite number
of variables, write
$\|P\|_{\mathrm{coef},1}=\sum_\alpha|p_\alpha|$. This norm is
submultiplicative. Indeed, for $Q=\sum_\beta q_\beta x^\beta$, expansion
and the triangle inequality give
\begin{equation}\label{gub:hd:detail:coefficient-product}
 \|PQ\|_{\mathrm{coef},1}
 \le\sum_{\alpha,\beta}|p_\alpha||q_\beta|
 =\|P\|_{\mathrm{coef},1}\|Q\|_{\mathrm{coef},1}.
\end{equation}

\subsection{Antipodal completion and Jacobi nodes}
Write $n_0=q_M+1$ for the number of original circle points,
$n=n_{\theta,M}$ for the number after antipodal completion, and
$k=k_M^{\mathrm{lat}}$ for the latitude count. The bounds $n,k\asymp M$
hold for the integer-valued angular degrees $q_M\asymp M$ specified in
Section~\ref{gub:sec:actual-dictionary}. The latitude count $k$ and the
signed-offset count $K$ refer to different sets of nodes.

For $p=3,\ldots,d$, set $a_p=(p-3)/2$ and let
$t_{p,1}<\cdots<t_{p,k}$ be the zeros of the Jacobi polynomial
$P_k^{(a_p,a_p)}$. Define their transverse scales by
\begin{equation}\label{gub:hd:latitude-nodes}
 P_k^{(a_p,a_p)}(t_{p,b})=0,
 \qquad s_{p,b}=(1-t_{p,b}^2)^{1/2},\quad 1\le b\le k.
\end{equation}
The uncompleted sets $Y^{(p)}$ and completed sets $X^{(p)}$ are defined
recursively by
\begin{align}
 Y^{(2)}&=\{(\cos(2\pi a/n_0),\sin(2\pi a/n_0)):0\le a<n_0\},\notag\\
 Y^{(p)}&=\{(s_{p,b}\eta,t_{p,b}):\eta\in Y^{(p-1)},\ 1\le b\le k\},\label{dir:Y}\\
 X^{(2)}&=\{(\cos(2\pi a/n),\sin(2\pi a/n)):0\le a<n\},\notag\\
 X^{(p)}&=\{(s_{p,b}\eta,t_{p,b}):\eta\in X^{(p-1)},\ 1\le b\le k\}.
 \label{gub:hd:completed-recursion}
\end{align}
The recursive formulas for $Y^{(p)}$ and $X^{(p)}$ apply for
$p=3,\ldots,d$, with the superscript denoting the ambient dimension.
Each lift lies on the unit sphere because $s_{p,b}^2+t_{p,b}^2=1$.
The following lemma identifies $X^{(p)}$ with the antipodal completion
of $Y^{(p)}$ and establishes that the interpolation nodes are the
prescribed direction nodes.

\begin{lemma}
\label{gub:hd:antipodal-completion}
For $p=2,\ldots,d$, all nodes at a given level are distinct, and
\begin{equation}\label{gub:hd:antipodal-identity}
 X^{(p)}=Y^{(p)}\cup(-Y^{(p)})=-X^{(p)},\qquad
 \#X^{(p)}=nk^{p-2}.
\end{equation}
In particular, $X^{(d)}$ is the actual direction set $X_M$ and
$J=nk^{d-2}/2$.
\end{lemma}
\begin{proof}
Jacobi polynomials have simple interior zeros. In the equal-parameter
case, they satisfy the parity identity
$P_k^{(a_p,a_p)}(-t)=(-1)^kP_k^{(a_p,a_p)}(t)$
\cite{Szego1975}.
Thus the roots are symmetric, and all $s_{p,b}$ are strictly positive.

At the circle level, when $n_0$ is even, rotation through $\pi$ moves each
point by exactly $n_0/2$ steps and leaves the set unchanged. Hence $n=n_0$.
When $n_0$ is odd, write the angles modulo $2\pi$ as integer multiples
of $\pi/n_0$. The original points occupy all even positions modulo
$2n_0$, and their antipodes occupy all odd positions. Their union
is the equally spaced set with $n=2n_0$ points, which proves the
identity at level two.

For the inductive step, define the lift of a set
$A\subset\mathbb S^{p-2}$ by
\[
 \mathcal L_p(A)=\{(s_{p,b}\eta,t_{p,b}):\eta\in A,\ 1\le b\le k\}.
\]
If $t_{p,b'}=-t_{p,b}$, then $s_{p,b'}=s_{p,b}$ and
\[
 -(s_{p,b}\eta,t_{p,b})=(s_{p,b'}(-\eta),t_{p,b'}).
\]
Symmetry of the latitude indices gives
\begin{equation}\label{gub:hd:detail:lift-commutation}
\begin{aligned}
 -\mathcal L_p(A)&=\mathcal L_p(-A),\\
 \mathcal L_p(A)\cup[-\mathcal L_p(A)]
 &=\mathcal L_p(A\cup(-A)).
 \end{aligned}
\end{equation}
Since $Y^{(p)}=\mathcal L_p(Y^{(p-1)})$ and
$X^{(p)}=\mathcal L_p(X^{(p-1)})$, the identity at level $p-1$
implies the identity at level $p$.

The last coordinate distinguishes different latitudes. Within each
latitude, multiplication by $s_{p,b}>0$ is injective, so all nodes are
distinct. Each lift therefore multiplies the number of nodes by exactly
$k$, and the initial $n$ circle points give $\#X^{(p)}=nk^{p-2}$.
Since no unit vector equals its antipode, each antipodal pair contains
two elements. The number of pairs is consequently half the node count.
\end{proof}

The recursive interpolation requires division by $s_{p,b}$ and
$t_{p,b}-t_{p,c}$. Since the polar latitudes approach the endpoints as
$k$ increases, the coefficient estimates must retain their dependence
on the degree. The following estimates bound the corresponding inverses
by $C_dk$ and $C_dk^2$.

\begin{lemma}
\label{gub:hd:jacobi-root-bounds}
For fixed $\lambda\ge0$ and each integer $k\ge1$, write the zeros of
$P_k^{(\lambda,\lambda)}$ as $\cos\varphi_j$, where
$0<\varphi_1<\cdots<\varphi_k<\pi$. There is $c_\lambda>0$, independent of $k$,
such that
\begin{equation}\label{gub:hd:angular-gap}
\begin{aligned}
 \sin\varphi_j&\ge c_\lambda/k &&(1\le j\le k),\\
 \varphi_{j+1}-\varphi_j&\ge c_\lambda/k &&(1\le j<k),\\
 |\cos\varphi_i-\cos\varphi_j|&\ge c_\lambda/k^2
 &&(1\le i,j\le k,\ i\ne j).
\end{aligned}
\end{equation}
\end{lemma}
\begin{proof}
We first bound the endpoint arcs. Set
$Y(\theta)=P_k^{(\lambda,\lambda)}(\cos\theta)/
P_k^{(\lambda,\lambda)}(1)$ and $z=(1-\cos\theta)/2$.
The Jacobi hypergeometric formula
\cite{Szego1975} gives
\[
 Y(\theta)=\sum_{\ell=0}^k
 \frac{(-k)_\ell(k+2\lambda+1)_\ell}
      {(\lambda+1)_\ell\,\ell!}\,z^\ell,
\]
where $(a)_\ell$ denotes the rising factorial. For $0\le\ell\le k$,
the numerator has modulus at most
$[(2+2\lambda)k^2]^\ell$, while
$(\lambda+1)_\ell\ge\ell!$. Consequently,
\[
 |Y(\theta)-1|
 \le\sum_{\ell\ge1}
       \frac{[(2+2\lambda)k^2z]^\ell}{(\ell!)^2}
 \le e^{(2+2\lambda)k^2z}-1.
\]
Since $z\le\theta^2/4$, a sufficiently small fixed
$c_\lambda>0$ makes this bound less than $1/2$ for
$0\le\theta\le c_\lambda/k$. Symmetry gives the same exclusion
near $\pi$. Thus
$\varphi_1,\pi-\varphi_k\ge c_\lambda/k$, and
$\sin\theta\ge c_\lambda'/k$ on $[\varphi_1,\varphi_k]$.

To bound the interior gaps, set
$U(\theta)=(\sin\theta)^{\lambda+1/2}
P_k^{(\lambda,\lambda)}(\cos\theta)$. The Jacobi differential
equation \cite{Szego1975} yields
\[
 U''+Q(\theta)U=0,\qquad
 Q(\theta)=\left(k+\lambda+\tfrac12\right)^2
       +\frac{\tfrac14-\lambda^2}{\sin^2\theta}.
\]
The endpoint bound implies $Q(\theta)\le C_\lambda k^2$ between
any two adjacent zeros. On such an interval, $U$ vanishes at both
ends, so integration by parts and the Dirichlet Poincar\'e inequality give
\[
 \frac{\pi^2}{(\varphi_{j+1}-\varphi_j)^2}\int |U|^2
 \le\int |U'|^2
 =\int Q|U|^2
 \le C_\lambda k^2\int |U|^2.
\]
Hence $\varphi_{j+1}-\varphi_j\ge c_\lambda''/k$.
Finally, for $1\le i<j\le k$,
\[
 \begin{aligned}
 |\cos\varphi_i-\cos\varphi_j|
 &=\int_{\varphi_i}^{\varphi_j}\sin\varphi\,\mathrm d\varphi\\
 &\ge\frac{c_\lambda'}{k}(\varphi_j-\varphi_i)\ge\frac{c_\lambda'''}{k^2}.
 \end{aligned}
\]
The first bound also covers $k=1$, for which the pairwise bounds
are vacuous. Taking the minimum of the constants proves
\eqref{gub:hd:angular-gap}.
\end{proof}

Applying the lemma with $\lambda=a_p$, we choose the constant uniformly
over the finitely many levels $p=3,\ldots,d$:
\begin{equation}\label{gub:hd:all-layer-bounds}
 s_{p,b}\ge c_d/k,\qquad
 |t_{p,b}-t_{p,c}|\ge c_d/k^2\quad(b\ne c).
\end{equation}

\subsection{Cardinal polynomials and homogenization}
Lemma~\ref{gub:hd:antipodal-completion} identifies $X^{(d)}$ with
the prescribed direction set $X_M$. We construct cardinal
polynomials on the recursive sets $X^{(p)}$.

\begin{lemma}
\label{gub:hd:cardinal-interpolation}
For each $p=2,\ldots,d$, define the integer
$D_p=(n-1)+(p-2)(k-1)$. For every $\xi\in X^{(p)}$ there is a real
polynomial $L_\xi^{(p)}:\R^p\to\R$ of degree at most $D_p$ such that
$L_\xi^{(p)}(\zeta)=\delta_{\xi\zeta}$ for $\zeta\in X^{(p)}$,
where $\delta_{\xi\zeta}$ is one if $\xi=\zeta$ and zero otherwise.
Moreover,
\begin{equation}\label{gub:hd:cardinal-coefficient-bound}
 \max_{\xi\in X^{(d)}}\|L_\xi^{(d)}\|_{\mathrm{coef},1}
 \le e^{C_dM\log(M+2)}.
\end{equation}
\end{lemma}
\begin{proof}
For $0\le a<n$, write $\theta_a=2\pi a/n$ and
$\xi_a=(\cos\theta_a,\sin\theta_a)$. Set
\begin{equation}\label{gub:hd:circle-cardinal}
 L_{\xi_a}^{(2)}(x_1,x_2)=\operatorname{Re}
 \left[\frac1n\sum_{h=0}^{n-1}e^{-\ii h\theta_a}(x_1+\ii x_2)^h\right].
\end{equation}
The real part is taken coefficientwise. At a circle node
$\xi_{a'}$, the bracketed expression is
\begin{equation}\label{gub:hd:detail:fourier-cardinality}
 \frac1n\sum_{h=0}^{n-1}
 \exp\!\left(\frac{2\pi\ii h(a'-a)}{n}\right)
 =\begin{cases}1,&a'=a,\\0,&a'\ne a.\end{cases}
\end{equation}
For $a'\ne a$, the ratio of this geometric series is a nontrivial $n$th
root of unity. Taking the real part preserves the cardinal values.
The polynomial has degree at most $n-1$. The binomial theorem gives
$\|(x_1+\ii x_2)^h\|_{\mathrm{coef},1}=2^h$. Multiplication by a
complex number of modulus one and taking real parts cannot increase
the coefficient norm. Hence
\begin{equation}\label{gub:hd:detail:circle-coefficients}
 \|L_{\xi_a}^{(2)}\|_{\mathrm{coef},1}
 \le\frac1n\sum_{h=0}^{n-1}2^h
 =\frac{2^n-1}{n}\le2^n.
\end{equation}

Assume that the polynomials have been constructed at level $p-1$,
and write $x=(x',x_p)\in\R^{p-1}\times\R$. For a target
$\xi=(s_{p,b}\eta,t_{p,b})$ with $\eta\in X^{(p-1)}$, set
\begin{equation}\label{gub:hd:cardinal-recursion}
\begin{aligned}
 L_\xi^{(p)}(x',x_p)&=\ell_{p,b}(x_p)
 L_\eta^{(p-1)}(x'/s_{p,b}),\\
 \ell_{p,b}(t)&=\prod_{\substack{1\le c\le k\\c\ne b}}
 \frac{t-t_{p,c}}{t_{p,b}-t_{p,c}}.
 \end{aligned}
\end{equation}
The scale $s_{p,b}$ is positive, depends only on the target node, and is
independent of $x'$. Hence \eqref{gub:hd:cardinal-recursion} defines a
polynomial. At a grid node $\zeta=(s_{p,c}\eta',t_{p,c})$, if $c\ne b$,
then $\ell_{p,b}(t_{p,c})=0$. If $c=b$, this factor equals one and
$s_{p,c}\eta'/s_{p,b}=\eta'$. The induction hypothesis then gives
\[
 L_\xi^{(p)}(\zeta)=L_\eta^{(p-1)}(\eta')
 =\delta_{\eta\eta'}.
\]
This proves the cardinal property at every node. Scaling preserves the
degree, and multiplication by $\ell_{p,b}$ increases it by $k-1$. Thus
$\deg L_\xi^{(p)}\le D_{p-1}+k-1=D_p$.

Let $B_p=\max\{1,\max_{\xi\in X^{(p)}}\|L_\xi^{(p)}\|_{\mathrm{coef},1}\}$.
For each latitude factor, \eqref{gub:hd:all-layer-bounds} and
submultiplicativity give
\begin{equation}\label{gub:hd:detail:latitude-coefficients}
 \|\ell_{p,b}\|_{\mathrm{coef},1}
 \le\prod_{c\ne b}\frac{1+|t_{p,c}|}{|t_{p,b}-t_{p,c}|}
 \le(C_dk^2)^{k-1}.
\end{equation}
If $L_\eta^{(p-1)}(x')=\sum_\alpha l_\alpha(x')^\alpha$,
then $0<s_{p,b}\le1$ and $|\alpha|\le D_{p-1}$ imply
\begin{equation}\label{gub:hd:detail:polar-scaling}
\begin{aligned}
 \|L_\eta^{(p-1)}(\,\cdot/s_{p,b})\|_{\mathrm{coef},1}
 &=\sum_\alpha|l_\alpha|s_{p,b}^{-|\alpha|}\\
 &\le s_{p,b}^{-D_{p-1}}\|L_\eta^{(p-1)}\|_{\mathrm{coef},1}
 \le(C_dk)^{D_{p-1}}B_{p-1}.
\end{aligned}
\end{equation}
The factor $(C_dk)^{D_{p-1}}$ accounts for the coefficient growth under
rescaling near the poles. With $C_d\ge2$, the coefficient bounds give
the recurrence
\begin{equation}\label{gub:hd:detail:coefficient-recurrence}
\begin{aligned}
 B_2&\le2^n,\\
 B_p&\le(C_dk^2)^{k-1}(C_dk)^{D_{p-1}}B_{p-1}.
\end{aligned}
\end{equation}
Taking logarithms and summing from $p=3$ to $d$ gives
\begin{equation}\label{gub:hd:detail:coefficient-log-sum}
\begin{aligned}
 \log B_d&\le n\log2\\
 &\quad+\sum_{p=3}^d\bigl[(k-1)\log(C_dk^2)+D_{p-1}\log(C_dk)\bigr].
\end{aligned}
\end{equation}
Here $n,k,D_p=O_d(M)$ and both logarithms are $O_d(T_M)$.
Each summand is $O_d(MT_M)$, and the number $d-2$ of summands is fixed.
Hence $\log B_d\le C_dMT_M$, which proves
\eqref{gub:hd:cardinal-coefficient-bound}. In dimension two, the sum is
empty and the conclusion follows directly from the circle estimate.
\end{proof}

Symmetrization combines the cardinal polynomials for each antipodal
pair. Homogenizing the resulting even monomials to a common degree
preserves their values on the unit sphere and gives the direction
selectors required in \eqref{dir:operator}.

\begin{proposition}[Homogeneous direction selectors]
\label{gub:prop:recursive-selector}\label{gub:ass:selectors}
For representatives $\omega_1,\ldots,\omega_J$ of the antipodal pairs,
there is a common even degree $r=2\lceil D_d/2\rceil=O_d(M)$ and real
homogeneous polynomials $P_1,\ldots,P_J$ such that
\begin{equation}\label{gub:hd:homogeneous-selector-bound}
\begin{aligned}
 P_\nu(\pm\omega_\mu)&=\delta_{\nu\mu}\quad(1\le\nu,\mu\le J),\\
 B_{\mathrm{sel}}&:=\max_\nu\|P_\nu\|_{\mathrm{coef},1}
 \le e^{C_dMT_M}.
\end{aligned}
\end{equation}
The selectors depend only on the prescribed directions, not on the
activation, offsets, target function, or output coefficients.
\end{proposition}
\begin{proof}
By Lemma~\ref{gub:hd:antipodal-completion}, the points
$\pm\omega_\mu$, $1\le\mu\le J$, are distinct nodes of $X^{(d)}$.
Using the cardinal polynomials of Lemma~\ref{gub:hd:cardinal-interpolation},
form an even polynomial by setting
\begin{equation}\label{gub:hd:detail:evenization}
 E_\nu(x)=L_{\omega_\nu}^{(d)}(x)+L_{\omega_\nu}^{(d)}(-x).
\end{equation}
At $\omega_\nu$, the first term is one and the second is zero.
Their values are interchanged at $-\omega_\nu$. Both terms vanish at
either member of every other antipodal pair. Thus
$E_\nu(\pm\omega_\mu)=\delta_{\nu\mu}$, and
\[
 \|E_\nu\|_{\mathrm{coef},1}
 \le2\|L_{\omega_\nu}^{(d)}\|_{\mathrm{coef},1}
 \le2e^{C_dMT_M}.
\]
Since $E_\nu(-x)=E_\nu(x)$, every coefficient of odd total degree
vanishes. Write
$E_\nu(x)=\sum_{|\alpha|\text{ even}}e_{\nu,\alpha}x^\alpha$ and set
\begin{equation}\label{gub:hd:homogenization}
 P_\nu(x)=\sum_{\substack{|\alpha|\le D_d\\|\alpha|\text{ even}}}
 e_{\nu,\alpha}x^\alpha(x_1^2+\cdots+x_d^2)^{(r-|\alpha|)/2}.
\end{equation}
The exponents of the added factors are nonnegative integers because
$r\ge D_d$ and both $r$ and $|\alpha|$ are even. Every summand has total
degree $r$. On the unit sphere, each added factor equals one. Hence
$P_\nu=E_\nu$ there, and the selector values are preserved. Since
$P_\nu(\omega_\nu)=1$, this is a nonzero homogeneous polynomial
of degree exactly $r$.

The coefficients of $(x_1^2+\cdots+x_d^2)^h$ are nonnegative. Their
sum is the value of this polynomial at $(1,\ldots,1)$, namely $d^h$, so
\begin{equation}\label{gub:hd:detail:homogenization-cost}
\begin{aligned}
 \|P_\nu\|_{\mathrm{coef},1}
 &\le\sum_\alpha|e_{\nu,\alpha}|d^{(r-|\alpha|)/2}\\
 &\le d^{r/2}\|E_\nu\|_{\mathrm{coef},1}
 \le2d^{r/2}e^{C_dMT_M}\le e^{C'_dMT_M}.
\end{aligned}
\end{equation}
The final inequality follows from $r=O_d(M)$, fixed $d$, and $T_M\ge1$.
\end{proof}

The recovery arguments in the next two sections use the exact selector
values, the common even degree, and the real coefficients.
The quantitative arguments use the bounds
$r=O_d(M)$ and $B_{\mathrm{sel}}\le e^{C_dMT_M}$.

\subsection{Planar selectors}
The preceding construction applies to the recovery estimates in every
dimension $d\ge2$. On the planar grid, a direct product yields an
alternative selector family of smaller common degree. Denote this degree by
$r_{\mathrm{pl}}$ to distinguish it from the degree fixed in
Proposition~\ref{gub:prop:recursive-selector}.

\begin{proposition}\label{gub:prop:selector-d2}
For the actual planar direction set, the selectors may be chosen with
a common even degree $r_{\mathrm{pl}}\le J$ and
$B_{\mathrm{sel}}\le e^{CM}$. They satisfy the same exact identities
$P_\nu(\pm\omega_\mu)=\delta_{\nu\mu}$ as in
Proposition~\ref{gub:prop:recursive-selector}.
\end{proposition}
\begin{proof}
The full circle grid has $n=2J$ points, and its unoriented directions
are equally spaced modulo $\pi$ with angular step $\pi/J$.
For any choice of representatives, let $v_\mu$ be $\omega_\mu$
rotated counterclockwise through $\pi/2$, and set
\begin{equation}\label{gub:hd:detail:planar-product}
 \mathcal Q_\nu(\xi)
 =\frac{\prod_{\substack{1\le\mu\le J\\\mu\ne\nu}}(v_\mu\cdot\xi)}
 {\prod_{\substack{1\le\mu\le J\\\mu\ne\nu}}(v_\mu\cdot\omega_\nu)}.
\end{equation}
The denominator is nonzero because the representatives belong to
distinct unoriented lines. At $\omega_\nu$, the numerator equals the
denominator. At any other $\omega_\mu$, the factor
$v_\mu\cdot\omega_\mu$ vanishes. Thus $\mathcal Q_\nu$ has the required
cardinal values at the representatives and is homogeneous of degree
$J-1$. For $J=1$, empty products are interpreted as one.

Changing the signs of representatives leaves the absolute value of the
denominator unchanged. The angular differences modulo $\pi$ range over
$\pi h/J$, $1\le h<J$.
Consequently,
\begin{equation}\label{gub:hd:detail:sine-product}
 \left|\prod_{\mu\ne\nu}(v_\mu\cdot\omega_\nu)\right|
 =\prod_{h=1}^{J-1}\sin\frac{\pi h}{J}
 =\frac{J}{2^{J-1}}.
\end{equation}
The last equality follows from the polynomial identity in the
complex variable $w$,
\[
 \frac{w^J-1}{w-1}
 =\prod_{h=1}^{J-1}(w-e^{2\pi\ii h/J}).
\]
At $w=1$ the left side has value $J$ by continuous extension, while
the modulus of the $h$th factor on the right is $2\sin(\pi h/J)$.

For a unit vector $v\in\R^2$, Cauchy-Schwarz gives
$\|v\cdot\xi\|_{\mathrm{coef},1}=|v_1|+|v_2|\le\sqrt2$.
Using \eqref{gub:hd:detail:coefficient-product} and
\eqref{gub:hd:detail:sine-product}, we obtain
\begin{equation}\label{gub:hd:detail:planar-coefficients}
 \|\mathcal Q_\nu\|_{\mathrm{coef},1}
 \le\frac{(\sqrt2)^{J-1}}{J/2^{J-1}}
 =\frac{(2\sqrt2)^{J-1}}{J}.
\end{equation}
If $J$ is odd, take $P_\nu=\mathcal Q_\nu$, whose degree
$r_{\mathrm{pl}}=J-1$ is already even. If $J$ is even, take
$P_\nu(\xi)=(\omega_\nu\cdot\xi)\mathcal Q_\nu(\xi)$, whose degree
is $r_{\mathrm{pl}}=J$. The additional factor equals one at the target
representative and preserves all other zeros. The coefficient norm
increases by a factor of at most $\sqrt2$. The even degree ensures
$P_\nu(-\omega_\mu)=P_\nu(\omega_\mu)$. Finally, $J=O(M)$
and \eqref{gub:hd:detail:planar-coefficients} give the stated
exponential bound.
\end{proof}

\subsection{Rotation and local coordinates}
\begin{lemma}\label{gub:lem:rotation}
For an orthogonal $d\times d$ matrix $Q$ and a homogeneous polynomial $P$
of degree $r$,
\[
 \|P(Q\,\cdot)\|_{\mathrm{coef},1}\le d^{r/2}\|P\|_{\mathrm{coef},1}.
\]
Under $x=x_0+Qz$ with $z\in\R^d$, one has
$P(\nabla_x)=P(Q\nabla_z)$, and volume is preserved.
\end{lemma}
\begin{proof}
Write the $i$th row of $Q$ as $(Q_{i1},\ldots,Q_{id})$. Its Euclidean
norm is one, so the linear form
$\ell_i(z)=\sum_jQ_{ij}z_j$ satisfies
$\|\ell_i\|_{\mathrm{coef},1}\le\sqrt d$. For each multi-index
$\alpha$ of degree $r$, submultiplicativity gives
\[
 \left\|\prod_{i=1}^d\ell_i^{\alpha_i}\right\|_{\mathrm{coef},1}
 \le\prod_{i=1}^d(\sqrt d)^{\alpha_i}=d^{r/2}.
\]
Multiplication by the absolute values of the coefficients and summation
over $\alpha$ prove the inequality.

The relation $z=Q^T(x-x_0)$ gives $\nabla_x=Q\nabla_z$.
Since partial derivatives commute, the same substitution holds for
polynomial differential operators. Finally, $|\det Q|=1$ preserves
volume and the corresponding $L^2$ integrals.
\end{proof}

Let $e_1=(1,0,\ldots,0)\in\R^d$. For each target direction, choose an
orthogonal matrix $Q_\nu$ with $Q_\nu e_1=\omega_\nu$. Then
$\omega_\nu\cdot Q_\nu z=z_1$, so the corresponding direction group
depends only on the first local coordinate. The selector identity
\eqref{dir:operator} holds in these coordinates, and the rotated
coefficient norm is at most $d^{r/2}B_{\mathrm{sel}}$.
Since $r=O_d(M)$, rotation increases the logarithmic bound by at most
$O_d(M)$.

The fixed interior ball in \eqref{gub:eq:domain-sandwich} contains the
observation region after rotation to any target direction. Set
$a=\min\{\rho/(2\sqrt d),1/2\}$. For every
orthogonal $Q$,
\begin{equation}\label{gub:hd:detail:fixed-observation-cube}
 x_0+Q[-a,a]^d\subset\overline B_{\rho/2}(x_0)\Subset\Omega,
\end{equation}
because $|Qz|=|z|\le\sqrt d\,a\le\rho/2$ on that cube.
For the tanh-difference argument, $\eps_{D,M}\to0$, so
$4\sqrt d\,\eps_{D,M}\le\rho$ for all sufficiently large $M$.
Hence
\begin{equation}\label{gub:hd:detail:scaled-observation-cube}
 x_0+\eps_{D,M}Q[-2,2]^d
 \subset\overline B_{\rho/2}(x_0)\Subset\Omega.
\end{equation}
Both cubes are centered at the fixed point $x_0$ and have distance at
least $\rho/2$ from $\R^d\setminus\Omega$. The next two sections use
test functions supported in these regions, with sufficient regularity
for integration by parts inside $\Omega$, to recover the coefficients
from the local $L^2$ norm of the ridge sum.

\section{Gaussian coefficient recovery}
\label{gub:g:section}

For the Gaussian feature family, set $\eps=\eps_{G,M}$ and consider
real resolutions $M\ge M_0\ge1$, with $M_0$ chosen so that
$0<\eps\le1$. Recovery for a single direction group starts with an
estimate on a fixed frequency band. Fourier estimates and tail bounds
then give a lower bound on a fixed large interval. The entire extension
of the Gaussian transfers this bound to an interior observation
interval. Homogeneous direction selectors and test functions supported
in the fixed interior ball recover all groups from the local $L^2$ norm
of the ridge sum.

In the signed-offset ordering \eqref{gub:eq:local-notation},
$J$ is the number of antipodal direction pairs, $K$ is the number of
signed offsets in each group, and $JK=N_M^{\mathrm{eq}}$.
The real centered offsets are
$\beta_{\nu j}=b_j-\omega_\nu\cdot x_0$. We use the bounds
\begin{equation}\label{gub:g:bias-assumption}
 \begin{gathered}
 1\le K\le C_KM,\qquad |\beta_{\nu j}|\le B_0,\\
 |\beta_{\nu i}-\beta_{\nu j}|\ge c_b/M\quad(i\ne j).
 \end{gathered}
\end{equation}
The bounds hold uniformly for $1\le\nu\le J$ and $1\le i,j\le K$,
with $i\ne j$ in the separation condition. The constants
$C_K,B_0,c_b>0$ are fixed. The centered midpoint offsets satisfy
\eqref{gub:g:bias-assumption} with the constants $C_K,B_0$ from
\eqref{gub:eq:shifted-bias-input} and $c_b=\kappa_0$.

For a coefficient vector $c=(c_{\nu j})\in\C^{JK}$, write
\begin{equation}\label{gub:g:dictionary}
\begin{aligned}
 F_\nu(t)&=\sum_{j=1}^K c_{\nu j}e^{-(t-\beta_{\nu j})^2/\eps^2},\\
 V_c(x)&=\sum_{\nu=1}^JF_\nu(\omega_\nu\cdot(x-x_0)).
 \end{aligned}
\end{equation}
Choose fixed positive constants $C_r,C_P,C_J$ such that the selectors
in Proposition~\ref{gub:ass:selectors} have common even degree
$r\le C_rM$ and coefficient bound $B_{\mathrm{sel}}\le e^{C_PMT_M}$,
and the number of groups satisfies $J\le C_JM^{d-1}$.

\subsection{Exponential sums on a fixed band}
Fix a direction group $\nu$ and write $\beta_j=\beta_{\nu j}$ for
$1\le j\le K$. For an arbitrary vector
$v=(v_j)_{j=1}^K\in\C^K$, set
\[
 S(\tau)=\sum_{j=1}^Kv_je^{-\ii\beta_j\tau},\qquad\tau\in\R.
\]

\begin{lemma}
\label{gub:g:fixed-band}
If the real offsets $\beta_1,\ldots,\beta_K$ satisfy
\eqref{gub:g:bias-assumption}, then for every $v\in\C^K$,
\begin{equation}\label{gub:g:fixed-band-bound}
 \|S\|_{L^2(1,2)}\ge e^{-CKT_M}\|v\|_2.
\end{equation}
\end{lemma}
\begin{proof}
Fix $A_0\ge\max\{2,2B_0\}$. For $1\le j\le K$, put
\[
 \begin{aligned}
 \tau_\ell&=1+\frac{\ell}{A_0K},\quad0\le\ell<K,\\
 z_j&=e^{-\ii\beta_j/(A_0K)},\qquad u_j=v_je^{-\ii\beta_j}.
 \end{aligned}
\]
The sample points lie in $[1,2]$ and satisfy
$S(\tau_\ell)=\sum_{j=1}^Ku_jz_j^\ell$.
Choose a fixed $c_{\mathrm{sep}}>0$ with
$c_{\mathrm{sep}}\le\min\{1,2c_b/(\pi A_0)\}$ and set
$\delta_M=c_{\mathrm{sep}}/(MK)$, so $0<\delta_M\le1$.
The absolute phase differences are bounded by
$2B_0/(A_0K)\le1$. Applying the chord estimate
$|e^{\ii s}-e^{\ii t}|\ge(2/\pi)|s-t|$ for $|s-t|\le\pi$ gives
\begin{equation}\label{gub:g:detail:band-separation}
 |z_i-z_j|\ge\frac{2c_b}{\pi A_0MK}\ge\delta_M
 \qquad(i\ne j).
\end{equation}
The nodes are distinct. For $1\le j\le K$, define the
Lagrange polynomial in the complex variable $z$ by
\[
 L_j(z)=\prod_{\substack{1\le i\le K\\i\ne j}}\frac{z-z_i}{z_j-z_i}
       =\sum_{\ell=0}^{K-1}a_{j\ell}z^\ell,
\]
so that $L_j(z_i)=\delta_{ji}$. Substitution of the sampled values gives
the inverse identity
\begin{equation}\label{gub:g:detail:band-inversion}
 \sum_{\ell=0}^{K-1}a_{j\ell}S(\tau_\ell)
 =\sum_{i=1}^Ku_iL_j(z_i)=u_j.
\end{equation}
Since $|z_i|=1$, each numerator factor has coefficient norm
$1+|z_i|=2$.
The coefficient-product bound \eqref{gub:hd:detail:coefficient-product}
gives the inverse Vandermonde estimate
\cite[Theorem~1]{Gautschi1962Vandermonde}:
\[
 \sum_{\ell=0}^{K-1}|a_{j\ell}|
 \le\prod_{i\ne j}\frac{1+|z_i|}{|z_j-z_i|}
 \le(2/\delta_M)^{K-1}.
\]
Since $|u_j|=|v_j|$, we have $\|u\|_2=\|v\|_2$.
Applying the coefficient bound in \eqref{gub:g:detail:band-inversion}
to an index with $|u_j|\ge K^{-1/2}\|u\|_2$ gives
\[
 \max_\ell|S(\tau_\ell)|
 \ge K^{-1/2}(\delta_M/2)^{K-1}\|v\|_2
 \ge e^{-C_1KT_M}\|v\|_2.
\]
Here $C_1>0$ is fixed. Taking logarithms and using
$\log(MK)\le CT_M$, which follows from $K\le C_KM$, proves the last
inequality after absorbing $K^{-1/2}$ into the exponential. Empty
products have value one, so the argument also covers $K=1$.

To obtain an integral lower bound, assume $v\ne0$ and set
$a_*=e^{-C_1KT_M}\|v\|_2$. At some $\tau_*\in[1,2]$,
$|S(\tau_*)|\ge a_*$. Differentiating the exponential sum gives the
uniform Lipschitz bound
\[
 |S'(\tau)|\le\sum_{j=1}^K|\beta_jv_j|
 \le B_0\sqrt K\|v\|_2\le C_2\sqrt K\|v\|_2,
\]
where $C_2\ge1$ is fixed.
Set $\ell_*=a_*/(2C_2\sqrt K\|v\|_2)$, so
$0<\ell_*\le1/2$. On the intersection of $[1,2]$ with the interval
of radius $\ell_*$ about $\tau_*$, the derivative bound gives
\[
 |S(\tau)|\ge |S(\tau_*)|-C_2\sqrt K\|v\|_2|\tau-\tau_*|
 \ge a_*/2.
\]
This intersection has length at least $\ell_*$. Consequently,
\[
 \|S\|_{L^2(1,2)}^2
 \ge\frac{a_*^2\ell_*}{4}
 =\frac{e^{-3C_1KT_M}}{8C_2\sqrt K}\|v\|_2^2.
\]
Absorbing $8C_2\sqrt K$ and taking a square root proves
\eqref{gub:g:fixed-band-bound}. The case $v=0$ is immediate.
\end{proof}

\subsection{Recovery on the real line and a fixed interval}
Retain the offsets $\beta_j$ and coefficient vector $v$ from the
preceding subsection, and define
\begin{equation}\label{gub:g:one-group}
 F(t)=\sum_{j=1}^Kv_je^{-(t-\beta_j)^2/\eps^2},\qquad t\in\R.
\end{equation}
We use the Fourier transform
$\widehat f(\tau)=\int_\R f(t)e^{-\ii t\tau}\dd t$ and the Plancherel identity
$\|f\|_{L^2(\R)}^2=(2\pi)^{-1}\|\widehat f\|_{L^2(\R)}^2$.
Translation and scaling of the classical Gaussian transform
\cite[Chapter~5]{SteinShakarchi2003} give
\begin{equation}\label{gub:g:gaussian-transform}
 \widehat F(\tau)=\eps\sqrt\pi\,e^{-\eps^2\tau^2/4}S(\tau).
\end{equation}

\begin{lemma}
\label{gub:g:whole-line}\label{gub:g:fixed-large-interval}
Assume \eqref{gub:g:bias-assumption} and set $L_*=B_0+2$.
There is a fixed $C_0>0$ such that, for every integer $r\ge0$,
\begin{equation}\label{gub:g:whole-line-bound}
 \|F^{(r)}\|_{L^2(\R)}\ge e^{-C_0KT_M}\|v\|_2.
\end{equation}
For integers $0\le r\le C_rM$, we also have
\begin{equation}\label{gub:g:fixed-large-bound}
 \|F^{(r)}\|_{L^2(-L_*,L_*)}\ge\tfrac12e^{-C_0KT_M}\|v\|_2.
\end{equation}
This second bound holds for all sufficiently large $M$, with a threshold
depending only on $C_K,B_0,c_b$, and $C_r$.
\end{lemma}
\begin{proof}
\emph{Step 1: obtain the lower bound on the whole line.}
Finite Gaussian sums and all their derivatives are Schwartz functions,
so Plancherel and
$\widehat{F^{(r)}}=(\ii\tau)^r\widehat F$ give
\[
 \begin{aligned}
 \|F^{(r)}\|_{L^2(\R)}^2
 &=\frac{\eps^2}{2}\int_\R |\tau|^{2r}
 e^{-\eps^2\tau^2/2}|S(\tau)|^2\dd\tau\\
 &\ge\frac{\eps^2e^{-2}}2\|S\|_{L^2(1,2)}^2.
 \end{aligned}
\]
Here $0<\eps\le1$ and $|\tau|^r\ge1$ on the retained band.
Taking a square root and applying Lemma~\ref{gub:g:fixed-band} gives
\[
 \|F^{(r)}\|_{L^2(\R)}
 \ge\frac{e^{-1}}{\sqrt2}\eps e^{-CKT_M}\|v\|_2.
\]
Since $\log(1/\eps)\le T_M$ and $K\ge1$, we enlarge $C_0$ to absorb
the fixed factor and $\eps$. Thus \eqref{gub:g:whole-line-bound} holds
with $C_0$ independent of $r$.

\emph{Step 2: bound the derivative tails.}
For real $t,\beta$, write $z=t+u+\ii w$ on the complex circle
$|z-t|=\eps$, with $u,w\in\R$ and $u^2+w^2=\eps^2$. The identity
\[
 (t-\beta+u)^2-\bigl((t-\beta)^2/2-u^2\bigr)
 =\tfrac12(t-\beta+2u)^2\ge0
\]
implies
$\operatorname{Re}((z-\beta)^2)\ge(t-\beta)^2/2-\eps^2$.
For every integer $l\ge0$, Cauchy's derivative estimate
\cite[Chapter~4]{Ahlfors1979} yields
\begin{equation}\label{gub:g:real-derivative-bound}
 \left|\frac{\dd^l}{\dd t^l}e^{-(t-\beta)^2/\eps^2}\right|
 \le e\,l!\eps^{-l}e^{-(t-\beta)^2/(2\eps^2)}.
\end{equation}
For $|t|>L>B_0$, the inequalities
$|t-\beta_j|\ge|t|-B_0$ and $\sum_j|v_j|\le\sqrt K\|v\|_2$ give
\[
 |F^{(r)}(t)|\le e\sqrt K\,r!\eps^{-r}
 e^{-(|t|-B_0)^2/(2\eps^2)}\|v\|_2.
\]
Set $A_L=(L-B_0)/\eps$. After squaring, integration over both tails gives
\[
 \|F^{(r)}\|_{L^2(|t|>L)}^2
 \le 2e^2K(r!)^2\eps^{1-2r}\|v\|_2^2
       \int_{A_L}^{\infty}e^{-u^2}\dd u.
\]
If $L-B_0\ge2$, then $A_L\ge2$ because $\eps\le1$, and
\[
 \begin{aligned}
 \int_{A_L}^\infty e^{-u^2}\dd u
 &\le\frac{1}{A_L}\int_{A_L}^\infty u e^{-u^2}\dd u\\
 &=\frac{e^{-A_L^2}}{2A_L}\le Ce^{-A_L^2/2}.
 \end{aligned}
\]
Taking a square root yields
\begin{equation}\label{gub:g:tail-bound}
 \|F^{(r)}\|_{L^2(|t|>L)}
 \le C\sqrt K\,r!\eps^{1/2-r}
 e^{-(L-B_0)^2/(4\eps^2)}\|v\|_2.
\end{equation}

\emph{Step 3: absorb the tail uniformly.}
For $0\le r\le C_rM$, the bounds $K\le C_KM$,
$r!\le(r+1)^r$, and $\log(1/\eps)\le T_M$ give, for a fixed
$C_{\mathrm{tail}}>0$,
\[
 \begin{aligned}
 \log\bigl(C\sqrt K\,r!\eps^{1/2-r}\bigr)
 &\le C+\tfrac12\log K+r\log(r+1)+rT_M\\
 &\le C_{\mathrm{tail}}MT_M.
 \end{aligned}
\]
With $L_*=B_0+2$ and $\eps^{-2}=M^2/T_M$,
\eqref{gub:g:tail-bound} gives
\begin{equation}\label{gub:g:detail:tail-budget}
 \|F^{(r)}\|_{L^2(|t|>L_*)}
 \le e^{C_{\mathrm{tail}}MT_M-M^2/T_M}\|v\|_2.
\end{equation}
Because $KT_M\le C_KMT_M$ and
$(M^2/T_M)/(MT_M)=M/T_M^2\to\infty$, we choose the common threshold
$M_0$ sufficiently large, independently of $r,K,v$ and the offsets,
to ensure that
\[
 C_{\mathrm{tail}}MT_M-M^2/T_M\le-C_0KT_M-\log2.
\]
The tail norm is at most one half of the whole-line lower bound.
The triangle inequality
\[
 \|F^{(r)}\|_{L^2(\R)}
 \le\|F^{(r)}\|_{L^2(-L_*,L_*)}
    +\|F^{(r)}\|_{L^2(|t|>L_*)}
\]
then proves \eqref{gub:g:fixed-large-bound}.
\end{proof}

The auxiliary interval $[-L_*,L_*]$ may extend beyond the centered
projection $\{\omega_\nu\cdot(x-x_0):x\in\Omega\}$. To obtain a lower
bound within the observation region, the next subsection uses the
entire extension of $F^{(r)}$ to propagate the estimate to a fixed
interior interval.

\subsection{Recovery on an interior interval}
Taylor approximation and the following polynomial propagation estimate
transfer the lower bound on $[-L_*,L_*]$ to a smaller fixed interval.
\begin{lemma}[Polynomial propagation]
\label{gub:lem:poly}
Let $I\subset\R$ be a fixed nondegenerate compact interval and
$\mathcal K\subset\C$ a fixed bounded set. There is $A\ge2$, depending
only on these sets, such that, for every integer $n\ge0$ and every complex
polynomial $p$ of degree at most $n$,
\begin{equation}\label{gub:g:poly-bound}
 \sup_{z\in\mathcal K}|p(z)|\le A^{n+1}\|p\|_{L^2(I)}.
\end{equation}
\end{lemma}
\begin{proof}
A fixed affine change of variables maps $I$ to $[-1,1]$ and maps
the set $\mathcal K$ into a disk $|z|\le B$, with fixed
$B>0$. Write $\widetilde p$ for the transformed polynomial.
Let $P_k$ denote the Legendre polynomial of degree $k$, normalized by
$P_k(1)=1$. Expand $\widetilde p$ in the orthonormal basis
$\sqrt{(2k+1)/2}\,P_k$. The bound $|P_k(t)|\le1$ on $[-1,1]$
\cite{Szego1975} and Cauchy-Schwarz give
\[
 \begin{aligned}
 \|\widetilde p\|_{L^\infty(-1,1)}
 &\le\left(\sum_{k=0}^n\frac{2k+1}{2}\right)^{1/2}
       \|\widetilde p\|_{L^2(-1,1)}\\
 &=\frac{n+1}{\sqrt2}\|\widetilde p\|_{L^2(-1,1)}.
 \end{aligned}
\]
The Bernstein-Walsh lemma \cite{RahmanSchmeisser2002}
bounds the supremum on this disk by
$C_B^n\|\widetilde p\|_{L^\infty(-1,1)}$, where $C_B\ge1$ depends only on $B$.
Absorbing $n+1$ into an exponential base and restoring the fixed
Jacobian factor proves \eqref{gub:g:poly-bound} for complex polynomials.
\end{proof}

Applied to a Taylor polynomial, the propagation estimate multiplies the
approximation error by $A^{n+1}$. We choose the fixed complex radius
large enough to preserve geometric decay after propagation, and then
choose the Taylor order.

\begin{lemma}\label{gub:g:local-recovery}
Fix $0<a\le1/2$ and put $I_0=[-a/2,a/2]$. Under
\eqref{gub:g:bias-assumption}, for integers $0\le r\le C_rM$ and all
sufficiently large $M$,
\begin{equation}\label{gub:g:local-bound}
\begin{aligned}
 \|F^{(r)}\|_{L^2(I_0)}
 &\ge\exp\{-C[\eps^{-2}+KT_M+r\log(r+2)]\}\|v\|_2\\
 &\ge e^{-C'(\eps^{-2}+MT_M)}\|v\|_2.
 \end{aligned}
\end{equation}
The constants and the threshold depend only on $a,C_K,B_0,c_b$, and $C_r$.
\end{lemma}
\begin{proof}
The case $v=0$ is immediate. Otherwise, normalize $\|v\|_2=1$.
Extend $F$ to $\C$ by the same finite Gaussian sum, and set
$H(z)=F^{(r)}(z)$. Then $H$ is entire. Put $I_*=[-L_*,L_*]$ and
$U=\|H\|_{L^2(I_0)}$. Apply Lemma~\ref{gub:lem:poly} with $I=I_0$
and $\mathcal K=I_*$ to fix $A\ge2$ such that, for every integer
$n\ge0$ and every complex polynomial $p$ of degree at most $n$,
\[
 \sup_{I_*}|p|\le A^{n+1}\|p\|_{L^2(I_0)}.
\]
This $A$ depends only on $a$ and $L_*$. Set $R_c=4AL_*$, which is
independent of $M,r$, and the truncation order.

\emph{Step 1: bound the derivative on the fixed complex disk.}
For $|z|\le R_c$, apply Cauchy's derivative formula on the circle of
radius one centered at $z$. At every point $w$ on that circle,
\[
 \begin{aligned}
 \left|e^{-(w-\beta_j)^2/\eps^2}\right|
 &=\exp\!\left(\frac{(\operatorname{Im}w)^2
                         -(\operatorname{Re}w-\beta_j)^2}{\eps^2}\right)\\
 &\le e^{(R_c+1)^2/\eps^2}.
 \end{aligned}
\]
Cauchy's formula and $\sum_j|v_j|\le\sqrt K$ give
\begin{equation}\label{gub:g:complex-growth}
 \sup_{|z|\le R_c}|H(z)|
 \le\sqrt K\,r!e^{(R_c+1)^2/\eps^2}=:B_r,
\end{equation}
and hence
\begin{equation}\label{gub:g:detail:complex-budget}
 \log B_r\le C\eps^{-2}+r\log(r+2)+C\log(K+1).
\end{equation}

\emph{Step 2: estimate the remainder after propagation.}
For an integer $n\ge0$, let $p_n$ be the Taylor polynomial of $H$
at zero through order $n$. Cauchy's coefficient estimate on
$|z|=R_c$ bounds the coefficient of $z^k$ by $B_rR_c^{-k}$ for every
integer $k\ge0$. Thus
\[
 \begin{aligned}
 \delta_n:=\sup_{I_*}|H-p_n|
 &\le B_r\sum_{k=n+1}^\infty(L_*/R_c)^k\\
 &\le2B_r(4A)^{-(n+1)}.
 \end{aligned}
\]
Since $I_0\subset I_*$ and $|I_0|=a$, the triangle inequality gives
$\|p_n\|_{L^2(I_0)}\le U+\sqrt a\,\delta_n$. Polynomial propagation
and the remainder bound therefore yield
\begin{equation}\label{gub:g:taylor-master}
\begin{aligned}
 \sup_{I_*}|H|
 &\le A^{n+1}U+(\sqrt a\,A^{n+1}+1)\delta_n\\
 &\le A^{n+1}U+B_r4^{-n}.
 \end{aligned}
\end{equation}
In the last inequality, the coefficient of $B_r4^{-n}$ is bounded by
$\sqrt a/2+1/(2A^{n+1})\le1$, because $a\le1/2$ and $A\ge2$.
Thus the radius $R_c=4AL_*$ preserves the geometric factor $4^{-n}$
after propagation. We choose $n$ so that this error is smaller than the
large-interval lower bound.

\emph{Step 3: choose the Taylor order and recover the local norm.}
By Lemma~\ref{gub:g:fixed-large-interval} and
$\|H\|_{L^2(I_*)}\le\sqrt{2L_*}\sup_{I_*}|H|$, we have
\[
 \sup_{I_*}|H|\ge b_*:=\frac{e^{-C_0KT_M}}{2\sqrt{2L_*}}.
\]
Choose the truncation order
\[
 n=\max\left\{0,\left\lceil\frac{\log(2B_r/b_*)}{\log4}\right\rceil\right\}.
\]
Then $B_r4^{-n}\le b_*/2$, and \eqref{gub:g:detail:complex-budget} gives
\begin{equation}\label{gub:g:detail:taylor-degree}
\begin{aligned}
 n+1&\le C[\eps^{-2}+KT_M+r\log(r+2)]\\
 &\le C'(\eps^{-2}+MT_M).
\end{aligned}
\end{equation}
The term $\log(K+1)$ and the fixed constants are absorbed into $KT_M$, using
$K\ge1$ and $T_M\ge\log3$. The second inequality uses
$K\le C_KM$ and $r\le C_rM$. Equation~\eqref{gub:g:taylor-master} gives
$b_*\le A^{n+1}U+b_*/2$, whence
\[
 U\ge\tfrac12b_*A^{-(n+1)}
 \ge\exp\{-C[\eps^{-2}+KT_M+r\log(r+2)]\}.
\]
Rescaling by the original $\|v\|_2$ proves the assertion.
\end{proof}

Since $\eps^{-2}=M^2/T_M$ and $MT_M=o(M^2/T_M)$, the quantity
$\eps^{-2}+MT_M$ in \eqref{gub:g:local-bound} has order $M^2/T_M$.

\subsection{Coefficient recovery on the domain}
A direction selector isolates the $r$th derivative of a single group.
Multiplying this derivative by a cutoff supported inside $\Omega$ gives
a test function. Integration by parts transfers the selector to this
function and relates the local recovery estimate to
$\|V_c\|_{L^2(\Omega)}$.

\begin{lemma}[Compactly supported cutoffs]\label{gub:g:cutoff}
For an integer $r\ge0$, define
$\theta_r(u)=(1-u^2)_+^{r+1}$ for $u\in\R$ and
$Z_r=\int_{-1}^1\theta_r(u)\dd u>0$, where $u_+=\max\{u,0\}$.
For fixed $a>0$ and $(t,y)\in\R\times\R^{d-1}$, put
\begin{equation}\label{gub:g:cutoff-definition}
\begin{aligned}
 \eta_r(t)&=\theta_r(t/a),\\
 \chi_r(y)&=\prod_{l=1}^{d-1}\frac{\theta_r(y_l/a)}{aZ_r},\\
 W_r(t,y)&=\eta_r(t)\chi_r(y).
 \end{aligned}
\end{equation}
Then $\chi_r\ge0$, $\int_{\R^{d-1}}\chi_r(y)\dd y=1$, and
\begin{equation}\label{gub:g:cutoff-positive}
 \eta_r(t)\ge(3/4)^{r+1}\quad(|t|\le a/2).
\end{equation}
The function $W_r$ belongs to $W^{r,\infty}(\R^d)$ and is supported
in $[-a,a]^d$. There is $C_{d,a}>0$, independent of $r$, such that
\begin{equation}\label{gub:g:cutoff-multivariate-bound}
 \max_{|\alpha|\le r}\|\partial^\alpha W_r\|_{L^\infty(\R^d)}
 \le e^{C_{d,a}(r+1)\log(r+2)}.
\end{equation}
\end{lemma}
\begin{proof}
On $|u|\le[2(r+1)]^{-1/2}$, Bernoulli's inequality gives
\[
 (1-u^2)^{r+1}\ge1-(r+1)u^2\ge\tfrac12.
\]
Integration over this interval gives
$Z_r\ge[2(r+1)]^{-1/2}$. The upper bound $Z_r\le2$ follows from
$0\le\theta_r\le1$. For the derivatives, expand on $[-1,1]$:
\[
 (1-u^2)^{r+1}
 =\sum_{j=0}^{r+1}(-1)^j\binom{r+1}{j}u^{2j}.
\]
For integers $0\le l\le r$, the terms with $2j<l$ vanish upon
differentiation. The differentiation factors in the remaining terms
are at most $(2r+2)^l$, and the binomial coefficients sum to $2^{r+1}$.
The polynomial has a zero of order $r+1$ at each endpoint, so its
derivatives through order $r$ agree there with those of the zero
extension. Repeated integration by parts against a smooth test function
produces no endpoint terms through order $r$. Hence the extended
function belongs to $W^{r,\infty}(\R)$ and satisfies
\begin{equation}\label{gub:g:detail:cutoff-one-dimensional}
 \|\theta_r^{(l)}\|_{L^\infty(\R)}\le2^{r+1}(2r+2)^l,
 \qquad0\le l\le r.
\end{equation}

Substitution $y_l=au$ in each transverse factor gives
$\int\chi_r=1$. The condition $|t/a|\le1/2$ gives
\eqref{gub:g:cutoff-positive}. For a multi-index $\alpha$ with
$|\alpha|\le r$, scaling and multiplying the one-dimensional estimates
give
\[
 \|\partial^\alpha W_r\|_{L^\infty(\R^d)}
 \le 2^{d(r+1)}\left(\frac{2r+2}{a}\right)^{|\alpha|}
 (aZ_r)^{-(d-1)}.
\]
Using $|\alpha|\le r$, the logarithm of this upper bound is at most
\[
 d(r+1)\log2+(r+d-1)\bigl[\log(2r+2)+|\log a|\bigr].
\]
For fixed $d,a$, this is at most $C_{d,a}(r+1)\log(r+2)$.
These estimates, together with the support and weak regularity of the
one-dimensional factors, prove \eqref{gub:g:cutoff-multivariate-bound}
and the stated support and regularity of the product.
\end{proof}

\begin{theorem}[Gaussian coefficient recovery]
\label{gub:g:coefficient-recovery}
Let $\eps=\eps_{G,M}$ and use the selectors of
Proposition~\ref{gub:ass:selectors}. Under
\eqref{gub:eq:domain-sandwich} and \eqref{gub:g:bias-assumption}, for all
sufficiently large real $M$ and every $c\in\C^{JK}$,
\begin{equation}\label{gub:g:coefficient-bound}
\begin{aligned}
 \|V_c\|_{L^2(\Omega)}
 &\ge e^{-C(\eps^{-2}+MT_M)}\|c\|_2\\
 &\ge e^{-C'M^2/T_M}\|c\|_2.
\end{aligned}
\end{equation}
The constants and the threshold are independent of $M$, the node
indices, and $c$. They may depend on the fixed domain, interior ball,
and construction parameters.
\end{theorem}
\begin{proof}
\emph{Step 1: select a direction group.}
For each $1\le\nu\le J$, set $D_\nu=P_\nu(\nabla_x)$. Write
$D_\nu^*$ for its formal $L^2$ adjoint. The coefficients of $P_\nu$
are real, and the common selector degree $r$ is even,
so $D_\nu^*=(-1)^rD_\nu=D_\nu$. For any smooth univariate $f$ and
$|\alpha|=r$, the chain rule gives
\[
 \partial_x^\alpha f(\omega_\mu\cdot(x-x_0))
 =\omega_\mu^\alpha f^{(r)}(\omega_\mu\cdot(x-x_0)).
\]
Since all monomials in $P_\nu$ have the same degree, we sum the chain-rule
identities with their coefficients and use $P_\nu(\omega_\mu)=\delta_{\nu\mu}$
to obtain
\begin{equation}\label{gub:g:selection-identity}
 D_\nu V_c=F_\nu^{(r)}(\omega_\nu\cdot(x-x_0)).
\end{equation}

\emph{Step 2: construct a test function inside the fixed ball.}
Fix $a=\min\{\rho/(2\sqrt d),1/2\}$ and choose an orthogonal matrix
$Q_\nu$ whose first column is $\omega_\nu$. The coordinates
$x=x_0+Q_\nu(t,y)$ satisfy $\omega_\nu\cdot(x-x_0)=t$.
Since $a\sqrt d\le\rho/2$, the rotated cube
$x_0+Q_\nu[-a,a]^d$ lies in $\overline B_{\rho/2}(x_0)$ and has
distance at least $\rho/2$ from $\R^d\setminus\Omega$.
The coordinate change has Jacobian of absolute value one.
Define a test function $\psi_\nu:\R^d\to\C$ and set
\[
 \begin{aligned}
 \psi_\nu(x)&=W_r(t,y)F_\nu^{(r)}(t),\\
 U_\nu&=\|F_\nu^{(r)}\|_{L^2(-a/2,a/2)},\\
 c_\nu&=(c_{\nu j})_{j=1}^K\in\C^K.
 \end{aligned}
\]
The transverse normalization and \eqref{gub:g:cutoff-positive} imply
\begin{equation}\label{gub:g:energy-positive}
\begin{aligned}
 \int_\Omega (D_\nu V_c)\overline{\psi_\nu}\dd x
 &=\int_\R\eta_r(t)|F_\nu^{(r)}(t)|^2\dd t\\
 &\ge e^{-C(r+1)}U_\nu^2.
 \end{aligned}
\end{equation}

\emph{Step 3: bound all derivatives of the test function.}
For integers $0\le l\le2r$, \eqref{gub:g:real-derivative-bound} and $\eps\le1$
give
\begin{align*}
 \|F_\nu^{(l)}\|_{L^\infty(\R)}
 &\le e\sqrt K\,l!\eps^{-l}\|c_\nu\|_2\\
 &\le e\sqrt K\,(2r)!\eps^{-2r}\|c_\nu\|_2
 \le e^{CMT_M}\|c_\nu\|_2.
\end{align*}
The final inequality follows from
\[
 \log((2r)!)\le2r\log(2r+1),\qquad
 2r\log(1/\eps)\le2rT_M,
\]
and $r\le C_rM$, $K\le C_KM$.

In the real coordinates $z=(t,y)\in\R^d$, the cutoff lemma and the
smoothness of $F_\nu^{(r)}$ justify the weak Leibniz rule for
$\psi_\nu(x_0+Q_\nu z)$. For $|\alpha|\le r$, it gives
\[
 \partial_z^\alpha[\psi_\nu(x_0+Q_\nu z)]
 =\sum_{\gamma\le\alpha}\binom{\alpha}{\gamma}
   (\partial_z^\gamma W_r)
   \partial_z^{\alpha-\gamma}F_\nu^{(r)}(t).
\]
The relation $\gamma\le\alpha$ is componentwise, and the sum of
the binomial coefficients is $2^{|\alpha|}\le2^r$. Terms with
any transverse derivative on $F_\nu^{(r)}(t)$ vanish. The remaining
terms contain derivatives of $F_\nu$ of orders between $r$ and $2r$.
Lemma~\ref{gub:g:cutoff} and the support volume $(2a)^d$ give
\begin{equation}\label{gub:g:detail:test-derivatives}
\begin{aligned}
 \|\partial_z^\alpha[\psi_\nu(x_0+Q_\nu z)]\|_{L^2(\R^d)}
 &\le (2a)^{d/2}2^r
      e^{C_{d,a}(r+1)\log(r+2)}e^{CMT_M}\|c_\nu\|_2\\
 &\le e^{C'MT_M}\|c_\nu\|_2.
\end{aligned}
\end{equation}
These derivatives have fixed compact support, and hence
$\psi_\nu\in W^{r,2}(\R^d)$.

Under the rotation, $\nabla_x=Q_\nu\nabla_z$, so the symbol of $D_\nu$
becomes $P_\nu(Q_\nu\,\cdot)$. By Lemma~\ref{gub:lem:rotation}, its
coefficient sum is at most
\[
 \|P_\nu(Q_\nu\,\cdot)\|_{\mathrm{coef},1}
 \le d^{r/2}B_{\mathrm{sel}}\le e^{CMT_M}.
\]
Combining this coefficient sum with the largest derivative norm in
\eqref{gub:g:detail:test-derivatives} gives a fixed
$C_{\mathrm{test}}>0$ such that
\begin{equation}\label{gub:g:adjoint-test-bound}
 \|D_\nu^*\psi_\nu\|_{L^2(\Omega)}
 \le e^{C_{\mathrm{test}}MT_M}\|c_\nu\|_2.
\end{equation}

\emph{Step 4: transfer the selector by integration by parts.}
Mollify $\psi_\nu$ using smooth mollifiers supported in balls whose radii
tend to zero and are less than $\rho/4$. The approximants are smooth,
have compact support in $\Omega$, and converge to $\psi_\nu$ in
$W^{r,2}(\R^d)$. For each approximant, $r$ integrations by parts give
the adjoint identity without boundary terms. Since $D_\nu V_c$ is
smooth on the common compact support and $D_\nu^*$ is a finite
constant-coefficient operator of order $r$, we may pass to the limit
and obtain
\begin{equation}\label{gub:g:detail:integration-by-parts}
 \int_\Omega(D_\nu V_c)\overline{\psi_\nu}\dd x
 =\int_\Omega V_c\overline{D_\nu^*\psi_\nu}\dd x.
\end{equation}
Combining this identity with \eqref{gub:g:energy-positive} and
\eqref{gub:g:adjoint-test-bound}, and applying Cauchy-Schwarz, yields
\[
 \begin{aligned}
 e^{-C(r+1)}U_\nu^2
 &\le\left|\int_\Omega V_c\overline{D_\nu^*\psi_\nu}\dd x\right|\\
 &\le\|V_c\|_{L^2(\Omega)}e^{C_{\mathrm{test}}MT_M}\|c_\nu\|_2.
 \end{aligned}
\]

\emph{Step 5: recover a maximal group and absorb the remaining factors.}
The case $c=0$ is immediate. Otherwise, choose a group of maximal
coefficient norm, so that
$\|c_\nu\|_2\ge J^{-1/2}\|c\|_2>0$.
Lemma~\ref{gub:g:local-recovery} gives a fixed $C_{\mathrm{loc}}>0$ such that
$U_\nu\ge e^{-C_{\mathrm{loc}}(\eps^{-2}+MT_M)}\|c_\nu\|_2$. Squaring this
bound and substituting into the preceding inequality gives
\[
 e^{-C(r+1)}e^{-2C_{\mathrm{loc}}(\eps^{-2}+MT_M)}\|c_\nu\|_2^2
 \le\|V_c\|_{L^2(\Omega)}e^{C_{\mathrm{test}}MT_M}\|c_\nu\|_2.
\]
After canceling the nonzero factor $\|c_\nu\|_2$, the bound
$r=O(M)$ and maximality give
\[
 \|V_c\|_{L^2(\Omega)}
 \ge e^{-C(\eps^{-2}+MT_M)}J^{-1/2}\|c\|_2.
\]
Since $J\le C_JM^{d-1}$, the contribution of the group factor to the
exponent is $\tfrac12\log J=O_d(T_M)$ and can be absorbed. Finally,
\[
 \eps^{-2}=\frac{M^2}{T_M},\qquad
 \frac{MT_M}{M^2/T_M}=\frac{T_M^2}{M}\longrightarrow0,
\]
which proves the second inequality in \eqref{gub:g:coefficient-bound}.
\end{proof}

\section{Tanh-difference coefficient recovery}
\label{gub:d2:section}

For the tanh-difference feature family, write $\eps=\eps_{D,M}$.
The domain, direction selectors, and centered-offset bounds in
\eqref{gub:g:bias-assumption} are the same as in the Gaussian case.
In the signed-offset indexing \eqref{gub:eq:local-notation}, there are $J$
direction representatives and $K$ signed offsets per group, with
$\beta_{\nu j}=b_j-\omega_\nu\cdot x_0$.
For $c=(c_{\nu j})\in\C^{JK}$, the ridge sum is
\begin{equation}\label{gub:d2:detail:physical-network}
 V_c(x)=\sum_{\nu=1}^J\sum_{j=1}^Kc_{\nu j}
 \sigma_D\!\left(\frac{\omega_\nu\cdot(x-x_0)-\beta_{\nu j}}{\eps}\right).
\end{equation}
The signed-offset reindexing permutes the original columns, so $\|c\|_2$
is the Euclidean norm of the original coefficient vector.

The one-dimensional argument uses the rational representation of
$\sigma_D$ and the period $\ii\pi$ of its meromorphic extension in the
unit-width variable.
A common denominator for the undifferentiated sum has degree determined
by the number of offsets, independently of the derivative order.

The argument starts with a lower bound for the distance of a nonzero
linear combination from the space of constants. For an integer $r\ge1$,
a small $r$th derivative norm gives a small Taylor approximation error
through order $r-1$.
Imaginary periodicity controls the nonconstant part of that polynomial,
so the distance estimate yields a lower bound for the derivative norm.
A compactly supported test function transfers this one-dimensional bound
to the $L^2(\Omega)$ norm. The local estimate holds for every integer
$r\ge0$. The direction selectors use orders $r=O(M)$.

\subsection{Distance from constants}
The single-group estimate allows general separated offsets in the
unit-width variable $t$. Fix an integer $K\ge1$, real scaled offsets
$a_1,\ldots,a_K$, and a vector
$v=(v_j)_{j=1}^K\in\C^K$. Set
\begin{equation}\label{gub:d2:one-group}
 F(t)=\sum_{j=1}^Kv_j\sigma_D(t-a_j),\qquad I=[-1,1].
\end{equation}
The estimate applies to any such offsets and parameters $\eps,A,h>0$
satisfying
\begin{equation}\label{gub:d2:scaled-assumptions}
 0<\eps\le1,\qquad |a_j|\le A/\eps,\qquad
 |a_i-a_j|\ge h/\eps\quad(i\ne j).
\end{equation}
Here $A$ bounds the offsets $\eps a_j$, and $h$ bounds their separation
from below. The corresponding separation in the unit-width variable is
at least $h/\eps$.
Define
\begin{equation}\label{gub:d2:E0}
 E_0=K\bigl[1+\eps^{-1}+\log(2+\eps/h)\bigr].
\end{equation}
Every direction group in the midpoint dictionary satisfies
\eqref{gub:d2:scaled-assumptions} for all sufficiently large $M$, with
$\eps=\eps_{D,M}$, $a_j=\beta_{\nu j}/\eps$, $v_j=c_{\nu j}$, $A=B_0$, and
$h=h_M\ge\kappa_0/M$. The scale bound follows from
\eqref{gub:eq:bounded-envelope}, and the offset bounds follow from
\eqref{gub:eq:signed-grid} and \eqref{gub:eq:shifted-bias-input}.

\begin{lemma}
\label{gub:d2:constant-distance}
There is $C_A>0$, depending only on $A$, such that under
\eqref{gub:d2:scaled-assumptions}, for every $v\in\C^K$ and $\gamma\in\C$,
\begin{equation}\label{gub:d2:constant-distance-bound}
 \|F-\gamma\|_{L^2(I)}\ge e^{-C_AE_0}\|v\|_2.
\end{equation}
\end{lemma}
\begin{proof}
\emph{Step 1: form a normalized denominator.}
Put $u_j=e^{2a_j-2}>0$ and $w_j=e^{2a_j+2}>0$.
The exponential formula for $\tanh$ expresses the sum as a rational
function of the complex variable $z$:
\begin{equation}\label{gub:d2:rational-form}
 F(t)=\mathcal R(e^{2t}),\qquad
 \mathcal R(z)=\sum_{j=1}^Kv_j\left(\frac{z-u_j}{z+u_j}
                          -\frac{z-w_j}{z+w_j}\right).
\end{equation}
Define the common denominator with normalization $Q_0(1)=1$:
\begin{equation}\label{gub:d2:simple-denominator}
 Q_0(z)=\prod_{j=1}^K\frac{z+u_j}{1+u_j}\frac{z+w_j}{1+w_j}.
\end{equation}
Retain repeated factors in the product and define
$P_\gamma(z)=Q_0(z)(\mathcal R(z)-\gamma)$, a polynomial of degree at most
$2K$ even when distinct summands share a pole. For each fraction in
\eqref{gub:d2:rational-form}, its denominator cancels one factor of
$Q_0$, leaving $2K-1$ linear factors multiplied by the linear numerator
and a constant normalization factor. The term $-\gamma Q_0$ has the
same degree bound.
On $\mathcal J=[e^{-2},e^2]$, each factor $(z+\lambda)/(1+\lambda)$,
with $\lambda=u_j$ or $w_j$, lies between $1$ and $z$. Consequently,
\begin{equation}\label{gub:d2:denominator-real}
 e^{-4K}\le Q_0(z)\le e^{4K}\qquad(z\in\mathcal J).
\end{equation}
Writing $W=\|F-\gamma\|_{L^2(I)}$, the substitution $z=e^{2t}$ gives
\begin{equation}
\begin{aligned}
 \|P_\gamma\|_{L^2(\mathcal J)}^2
 &=\int_{-1}^1|Q_0(e^{2t})|^2|F(t)-\gamma|^2\,2e^{2t}\dd t\\
 &\le 2e^2e^{8K}W^2.
\end{aligned}
 \label{gub:d2:detail:numerator-jacobian}
\end{equation}
Since the interval $\mathcal J$ and the disk $|z|\le1$ are fixed,
Lemma~\ref{gub:lem:poly} for degree $2K$ gives
\begin{equation}\label{gub:d2:numerator-bound}
 \sup_{|z|\le1}|P_\gamma(z)|\le e^{CK}W.
\end{equation}

\emph{Step 2: bound the Taylor coefficients near zero.}
Let
\begin{equation}\label{gub:d2:detail:small-radius}
\begin{aligned}
 \delta_0&=\tfrac12\min\{1,u_1,w_1,\ldots,u_K,w_K\},\\
 \tfrac12e^{-2A/\eps-2}&\le\delta_0\le\tfrac12.
\end{aligned}
\end{equation}
Every $\lambda$ occurring in $Q_0$ satisfies $\lambda\ge2\delta_0$.
Thus $|z+\lambda|\ge\lambda/2$ on $|z|\le\delta_0$, and
\begin{equation}\label{gub:d2:denominator-small}
\begin{aligned}
 \sup_{|z|\le\delta_0}|Q_0(z)|^{-1}
 &\le\prod_j2(1+u_j^{-1})\,2(1+w_j^{-1})\\
 &\le e^{C_AK(1+\eps^{-1})}.
\end{aligned}
\end{equation}
The last estimate uses $u_j^{-1},w_j^{-1}\le e^{2A/\eps+2}$
for all $2K$ factors. Every possible pole lies at least
$2\delta_0$ from the origin, so $\mathcal R-\gamma$ is analytic in a
neighborhood of the closed disk. Denote by $[z^n](\mathcal R-\gamma)$
the coefficient of $z^n$ in its Taylor expansion at zero.
Cauchy's coefficient formula
\cite[Chapter~4, Section~2.3]{Ahlfors1979},
\eqref{gub:d2:numerator-bound}, and
\eqref{gub:d2:denominator-small} give, for integers $1\le n\le K$,
\begin{equation}\label{gub:d2:positive-taylor-coefficients}
\begin{aligned}
 |[z^n](\mathcal R-\gamma)|
 &\le\delta_0^{-n}e^{C_AK(1+\eps^{-1})}W\\
 &\le e^{C_A'K(1+\eps^{-1})}W.
\end{aligned}
\end{equation}
The last inequality uses $\log(\delta_0^{-1})\le C_A(1+\eps^{-1})$
and $n\le K$.

\emph{Step 3: recover the moments of orders $1,\ldots,K$.}
Set $x_j=e^{-2a_j}>0$ and
$S_n=\sum_{j=1}^Kv_jx_j^n$ for integers $n\ge1$.
For $\lambda>0$ and $|z|<\lambda$, the geometric series gives
\begin{equation}\label{gub:d2:detail:geometric-expansion}
 \frac{z-\lambda}{z+\lambda}
 =-1+2\sum_{n=1}^{\infty}(-1)^{n-1}z^n\lambda^{-n}.
\end{equation}
Since $u_j^{-n}-w_j^{-n}=2e^{-2na_j}\sinh(2n)$, subtracting the two
series gives
\begin{equation}\label{gub:d2:moment-identity}
 [z^n](\mathcal R-\gamma)
 =4(-1)^{n-1}\sinh(2n)S_n,\qquad n\ge1.
\end{equation}
The constant $\gamma$ does not affect these positive-order coefficients.
Since $4\sinh(2n)>1$, \eqref{gub:d2:positive-taylor-coefficients} gives
\[
 |S_n|\le e^{C_A'K(1+\eps^{-1})}W,\qquad1\le n\le K.
\]

\emph{Step 4: invert the finite moment system.}
Set $\eta=\min\{1,h/\eps\}$. The nodes satisfy
\begin{equation}\label{gub:d2:detail:moment-node-separation}
\begin{aligned}
 e^{-2A/\eps}&\le x_j\le e^{2A/\eps},\\
 |x_i-x_j|&\ge2e^{-2A/\eps}\eta\qquad(i\ne j).
\end{aligned}
\end{equation}
The derivative of $a\mapsto e^{-2a}$ has absolute value at least
$2e^{-2A/\eps}$ on the allowed real interval. The mean value
theorem and $|a_i-a_j|\ge h/\eps\ge\eta$ give the separation bound.
For $1\le j\le K$, define the Lagrange polynomial in the variable $X$ by
$\ell_j(X)=\prod_{i\ne j}(X-x_i)/(x_j-x_i)
=\sum_{n=0}^{K-1}d_{jn}X^n$, so $\ell_j(x_i)=\delta_{ji}$.
Multiply the interpolation identities by $v_ix_i$ and sum over
$i$ to obtain
\begin{equation}\label{gub:d2:moment-inversion}
 v_jx_j=\sum_{n=0}^{K-1}d_{jn}S_{n+1}.
\end{equation}
Gautschi's inverse Vandermonde bound
\cite{Gautschi1962Vandermonde} gives
\begin{equation}
\begin{aligned}
 \sum_{n=0}^{K-1}|d_{jn}|
 &\le\prod_{i\ne j}\frac{1+x_i}{|x_j-x_i|}\\
 &\le\left(\frac{1+e^{2A/\eps}}{2e^{-2A/\eps}\eta}\right)^{K-1}\\
 &\le e^{C_AK(1+\eps^{-1})+K\log(2+\eps/h)}.
\end{aligned}
 \label{gub:d2:moment-inverse-bound}
\end{equation}
Here $\log(\eta^{-1})\le\log(2+\eps/h)$, and the logarithm of the
remaining part of each factor is at most $C_A(1+\eps^{-1})$. Equation
\eqref{gub:d2:moment-inversion} and the first $K$ moment bounds then give,
for every $j$,
\begin{equation}\label{gub:d2:detail:individual-coefficient}
\begin{aligned}
 |v_j|&\le e^{2A/\eps}
 \left(\sum_{n=0}^{K-1}|d_{jn}|\right)
 \max_{1\le n\le K}|S_n|\\
 &\le e^{C_AE_0}W.
\end{aligned}
\end{equation}
The Euclidean norm introduces a factor $\sqrt K\le e^{CK}$,
which can be absorbed into $C_A$. When $K=1$, the Lagrange product is
empty and equals one, so the argument still applies.
\end{proof}

\subsection{Entire-function propagation}
The imaginary-period argument needs bounds on both $I$ and
$I+\ii\pi$. Both sets lie in the disk $|t|\le L$ whenever
$L>\sqrt{1+\pi^2}$. The estimate below controls the Taylor remainder
after polynomial propagation.

\begin{lemma}
\label{gub:d2:entire-propagation}
Fix $L>\sqrt{1+\pi^2}$. There is $A_*\ge2$, depending only on $L$,
such that, with $R_*=8A_*L$, every entire function $H:\C\to\C$ with a
bound $B_H\ge0$ satisfying $\sup_{|t|\le R_*}|H(t)|\le B_H$ obeys,
for every integer $\ell\ge0$,
\begin{equation}\label{gub:d2:entire-propagation-bound}
 \sup_{|t|\le L}|H(t)|
 \le A_*^{\ell+1}\|H\|_{L^2(I)}+B_H4^{-\ell}.
\end{equation}
\end{lemma}
\begin{proof}
Lemma~\ref{gub:lem:poly}, applied to the fixed interval $I$ and the disk
$|t|\le L$, gives $A_*\ge2$ such that, for every integer $\ell\ge0$
and every complex polynomial $p$ of degree at most $\ell$,
\begin{equation}\label{gub:d2:detail:fixed-polynomial-propagation}
 \sup_{|t|\le L}|p(t)|\le A_*^{\ell+1}\|p\|_{L^2(I)}.
\end{equation}
Fix $R_*=8A_*L$. For any integer $\ell\ge0$, let $p_\ell$
be the Taylor polynomial of $H$ at zero through order $\ell$.
Cauchy's coefficient bounds on $|t|=R_*$ give
\begin{equation}
\begin{aligned}
 \delta_\ell:=\sup_{|t|\le L}|H(t)-p_\ell(t)|
 &\le B_H\sum_{n=\ell+1}^{\infty}(L/R_*)^n\\
 &\le2B_H(8A_*)^{-\ell-1}.
\end{aligned}
 \label{gub:d2:detail:entire-remainder}
\end{equation}
On the real interval, $\|p_\ell\|_{L^2(I)}\le
\|H\|_{L^2(I)}+\sqrt2\delta_\ell$. Applying polynomial propagation
and the triangle inequality gives
\[
 \sup_{|t|\le L}|H(t)|
 \le A_*^{\ell+1}\|H\|_{L^2(I)}
 +(\sqrt2A_*^{\ell+1}+1)\delta_\ell.
\]
The last term is at most $B_H4^{-\ell}$ because
\[
 2(\sqrt2+A_*^{-\ell-1})8^{-\ell-1}\le4^{-\ell}
 \qquad(\ell\ge0,\ A_*\ge2).
\]
\end{proof}

\subsection{Polynomial differences under an imaginary shift}
A polynomial with period $\ii\pi$ is constant. The estimate below
bounds the nonconstant part of a polynomial by its difference under
this shift, with explicit dependence on the degree.

\begin{lemma}
\label{gub:d2:polynomial-difference}
For an integer $n\ge0$ and a complex polynomial $p$ of degree at most $n$, let
$q_p(t)=p(t+\ii\pi)-p(t)$. There is an absolute constant $C>0$,
independent of $n$ and $p$, such that
\begin{equation}\label{gub:d2:difference-bound}
 \|p-p(0)\|_{L^\infty(I)}
 \le e^{C(n+1)\log(n+2)}\|q_p\|_{L^\infty(I)}.
\end{equation}
\end{lemma}
\begin{proof}
The assertion is immediate for $n=0$. For $n\ge1$, set $\tau=\ii\pi$.
Define the Bernoulli numbers $\mathsf B_k$ by the generating function
\cite{OlverEtAl2010NIST} in the auxiliary variable $w$:
\[
 \mathcal B(w)=\frac{w}{e^w-1}
 =\sum_{k=0}^{\infty}\frac{\mathsf B_k}{k!}w^k,\qquad |w|<2\pi,
\]
with the removable value $\mathcal B(0)=1$. Cauchy's coefficient bound
on the fixed circle $|w|=\pi$ gives
\begin{equation}\label{gub:d2:detail:bernoulli-coefficients}
 \frac{|\mathsf B_k|}{k!}\le C\pi^{-k},\qquad
 \frac{|\mathsf B_k\tau^k|}{k!}\le C\quad(k\ge0).
\end{equation}
Let $D=\dd/\dd t$ on the vector space of polynomials of degree at most
$n$. Since $D^{n+1}=0$ on this space, the power-series expressions
$e^{\tau D}$ and $\mathcal B(\tau D)$ reduce to finite sums.
Taylor's formula and the Bernoulli generating function give
\[
 q_p=(e^{\tau D}-1)p,
 \qquad \mathcal B(\tau D)(e^{\tau D}-1)p=\tau Dp.
\]
Substituting $\tau D$ into
$\mathcal B(w)(e^w-1)=w$ gives the second identity. Since $\deg q_p\le n-1$,
this identity reduces to
\begin{equation}\label{gub:d2:finite-difference-inverse}
 p'=\tau^{-1}\sum_{k=0}^{n-1}\frac{\mathsf B_k}{k!}\tau^kq_p^{(k)}.
\end{equation}
By Lemma~\ref{gub:lem:poly} and
$\|q_p\|_{L^2(I)}\le\sqrt2\|q_p\|_{L^\infty(I)}$, polynomial propagation
extends the bound for $q_p$ from $I$ to the disk $|t|\le2$ with a multiplicative factor
$C^{n+1}$. Each unit circle centered at $t\in I$ lies in this disk, so
Cauchy's derivative formula gives
\begin{equation}\label{gub:d2:detail:difference-derivatives}
 \|q_p^{(k)}\|_{L^\infty(I)}
 \le k!C^{n+1}\|q_p\|_{L^\infty(I)}
 \qquad(0\le k\le n-1).
\end{equation}
This bound, \eqref{gub:d2:detail:bernoulli-coefficients}, and
\eqref{gub:d2:finite-difference-inverse} give
\[
 \|p'\|_{L^\infty(I)}
 \le C^{n+2}\left(\sum_{k=0}^{n-1}k!\right)
                  \|q_p\|_{L^\infty(I)}.
\]
Since $\sum_{k=0}^{n-1}k!\le(n+1)n!$, the factor
$C^{n+2}\sum_{k=0}^{n-1}k!$ is at most
$e^{C'(n+1)\log(n+2)}$. Integrating $p'$ from zero to $t\in I$
over an interval of length at most one yields the asserted bound.
\end{proof}

\subsection{Local derivative lower bounds}
The lower bound for the distance of $F$ to constants in
Lemma~\ref{gub:d2:constant-distance} yields lower bounds for its
derivative norms. If a derivative of positive order is small in
$L^2(I)$, Taylor's formula makes $F$ close to a polynomial on $I$.
After the poles are cleared, Lemma~\ref{gub:d2:entire-propagation} can
propagate this approximation to $I+\ii\pi$. The imaginary period of
$F$ and Lemma~\ref{gub:d2:polynomial-difference} then show that the
polynomial is close to a constant. The distance lower bound therefore
limits how small the derivative norm can be.

\begin{theorem}
\label{gub:d2:local-recovery}
There is $C_A>0$, depending only on $A$, such that every $F$ in
\eqref{gub:d2:one-group} whose offsets satisfy
\eqref{gub:d2:scaled-assumptions} obeys, for every integer $r\ge0$,
\begin{equation}\label{gub:d2:local-recovery-bound}
 \|F^{(r)}\|_{L^2(I)}
 \ge e^{-C_A[E_0+r\log(r+2)]}\|v\|_2.
\end{equation}
\end{theorem}
\begin{proof}
Let $C_{\mathrm{dist}}>0$, depending only on $A$, be a constant from
Lemma~\ref{gub:d2:constant-distance}. The assertion is immediate for $v=0$.
Otherwise, normalize $\|v\|_2=1$. The case
$r=0$ follows from Lemma~\ref{gub:d2:constant-distance} with $\gamma=0$.
Suppose $r\ge1$. Fix $L=4$ and choose $A_*,R_*$ as in
Lemma~\ref{gub:d2:entire-propagation}. Both are absolute constants.
Set $U=\|F^{(r)}\|_{L^2(I)}$. For $U>1$, the asserted lower bound
already holds, so it suffices to consider
$U\le1$.

\emph{Control of the Taylor polynomial.}
Let $p$ be the Taylor polynomial of $F$ at zero through order $r-1$.
The integral remainder on the real interval gives
\[
 F(t)-p(t)=\frac1{(r-1)!}\int_0^t(t-u)^{r-1}F^{(r)}(u)\dd u.
\]
For either sign of $t$, Cauchy-Schwarz bounds the integral by $U$
times the square root of
$\int_0^{|t|}s^{2r-2}\dd s=|t|^{2r-1}/(2r-1)$. Thus
\begin{equation}\label{gub:d2:primitive-smallness}
 |F(t)-p(t)|
 \le\frac{|t|^{r-1/2}}{(r-1)!\sqrt{2r-1}}U\le U\quad(t\in I).
\end{equation}
Since $|\sigma_D|\le2$ on the real axis,
$\|p\|_{L^\infty(I)}\le2\sqrt K+1\le3\sqrt K$.
Applying Lemma~\ref{gub:lem:poly} on the fixed disk of radius $R_*$ from
Lemma~\ref{gub:d2:entire-propagation} gives
\begin{equation}\label{gub:d2:unknown-polynomial-complex}
 \sup_{|t|\le R_*}|p(t)|\le C_{R_*}^{r+1}\sqrt K.
\end{equation}
Here $C_{R_*}\ge2$ depends only on $R_*$ and may be enlarged below.

\emph{Removal of all poles and growth on a fixed disk.}
Extend $F$ meromorphically by the same finite sum and define
\begin{equation}\label{gub:d2:entire-function}
 H(t)=Q_0(e^{2t})[F(t)-p(t)].
\end{equation}
Since $Q_0(z)\mathcal R(z)$ is a polynomial, both
$Q_0(e^{2t})\mathcal R(e^{2t})$ and $Q_0(e^{2t})p(t)$ are entire.
Thus $H$ extends across all poles to an entire function.
Enlarge $C_{R_*}$ if necessary so that, for $|t|\le R_*$ and $z=e^{2t}$,
$|z\pm\lambda|/(1+\lambda)\le C_{R_*}$ for every $\lambda>0$.
In \eqref{gub:d2:rational-form}, cancel the denominator factor of
each summand before applying this bound. For each fraction, cancellation
leaves $2K$ normalized linear factors, including its numerator. The
bound also applies to repeated factors. Since
$\sum_j|v_j|\le\sqrt K$, we obtain
\begin{equation}\label{gub:d2:detail:cancelled-growth}
\begin{aligned}
 |Q_0(e^{2t})F(t)|&\le2\sqrt K C_{R_*}^{2K},\\
 |Q_0(e^{2t})|&\le C_{R_*}^{2K},
 \end{aligned}
\end{equation}
These bounds also hold at poles after the singularities are removed.
Together with \eqref{gub:d2:unknown-polynomial-complex}, they give a constant
$C_*>0$, depending only on $R_*$, such that
\begin{equation}\label{gub:d2:entire-growth}
 \sup_{|t|\le R_*}|H(t)|\le B_H:=e^{C_*(K+r)}.
\end{equation}

\emph{Propagation to the translated interval.}
On $I$, equations \eqref{gub:d2:denominator-real} and
\eqref{gub:d2:primitive-smallness} give
$\|H\|_{L^2(I)}\le\sqrt2e^{4K}U$. Lemma~\ref{gub:d2:entire-propagation}
then gives, for every integer $\ell\ge0$,
\begin{equation}\label{gub:d2:propagated-smallness}
 V_H:=\sup_{|t|\le L}|H(t)|
 \le\sqrt2e^{4K}A_*^{\ell+1}U+B_H4^{-\ell}.
\end{equation}

\emph{Reduction of the polynomial to a constant.}
For real $t$, both $t$ and $t+\ii\pi$ are regular points of $F$, since
the corresponding value of $e^{2t}$ is positive and all poles in the
exponential variable are negative. Imaginary periodicity gives
$F(t+\ii\pi)=F(t)$. Together with $e^{2(t+\ii\pi)}=e^{2t}$, this gives
\begin{equation}\label{gub:d2:detail:periodic-difference}
 H(t+\ii\pi)-H(t)
 =-Q_0(e^{2t})[p(t+\ii\pi)-p(t)].
\end{equation}
For $t\in I$, both arguments on the left lie in the disk of radius $L$.
The lower bound in
\eqref{gub:d2:denominator-real} therefore yields
\begin{equation}\label{gub:d2:detail:polynomial-difference-bound}
 \|p(\cdot+\ii\pi)-p\|_{L^\infty(I)}\le2e^{4K}V_H.
\end{equation}
Set $\gamma=p(0)$. The triangle inequality yields the bound
$\|F-\gamma\|_{L^2(I)}\le\sqrt2[U+\|p-\gamma\|_{L^\infty(I)}]$.
Lemma~\ref{gub:d2:polynomial-difference}, applied with degree at most
$r-1$, then gives
\begin{equation}\label{gub:d2:distance-from-periodicity}
 \|F-\gamma\|_{L^2(I)}
 \le CU+Ce^{4K+Cr\log(r+2)}V_H.
\end{equation}

\emph{Choice of truncation and absorption of the amplified error.}
Set $D_*=E_0+r\log(r+2)$, so that
$D_*\ge K\ge1$ and
$K+r\le CD_*$. Lemma~\ref{gub:d2:constant-distance}, applied to
$\gamma=p(0)$, together with
\eqref{gub:d2:propagated-smallness} and
\eqref{gub:d2:distance-from-periodicity}, gives
\begin{equation}\label{gub:d2:absorption-master}
\begin{aligned}
 e^{-C_{\mathrm{dist}}E_0}
 &\le CU\\
 &\quad{}+Ce^{4K+Cr\log(r+2)}
 \bigl[\sqrt2e^{4K}A_*^{\ell+1}U+B_H4^{-\ell}\bigr].
 \end{aligned}
\end{equation}
By \eqref{gub:d2:entire-growth}, the amplified remainder satisfies,
for some absolute constant $C_1>0$,
\begin{equation}\label{gub:d2:detail:amplified-remainder}
 Ce^{4K+Cr\log(r+2)}B_H4^{-\ell}
 \le e^{C_1D_*-\ell\log4}.
\end{equation}
The fixed prefactor can be absorbed into $C_1$ using $D_*\ge1$. Choose
$C_2>0$, depending only on $A$, such that
\begin{equation}\label{gub:d2:detail:truncation-choice}
 \ell=\lceil C_2D_*\rceil,\qquad
 C_2\log4\ge C_1+C_{\mathrm{dist}}+\log2.
\end{equation}
Since $E_0\le D_*$ and $D_*\ge1$, the remainder in
\eqref{gub:d2:detail:amplified-remainder} is then at most
$\tfrac12e^{-C_{\mathrm{dist}}E_0}$.
Absorbing the remainder into the left-hand side of
\eqref{gub:d2:absorption-master} gives
\begin{equation}
\begin{aligned}
 \tfrac12e^{-C_{\mathrm{dist}}E_0}
 &\le\left[C+C\sqrt2e^{8K+Cr\log(r+2)}A_*^{\ell+1}\right]U\\
 &\le e^{C_3D_*}U.
\end{aligned}
 \label{gub:d2:detail:absorption-finish}
\end{equation}
The second inequality follows from $\ell+1\le(C_2+2)D_*$ and the fixed value
of $A_*$. Thus $C_3$ depends only on $A$. Absorbing the factor $1/2$
gives $U\ge e^{-C_4D_*}$ for a constant $C_4$ depending only on $A$.
Choosing $C_A\ge\max\{C_{\mathrm{dist}},C_4\}$ covers both $r=0$ and
$r\ge1$. Rescaling $v$ completes the proof.
\end{proof}

\subsection{Coefficient recovery on the domain}
To transfer the one-dimensional recovery estimate to
$L^2(\Omega)$, we rescale the domain and use a direction selector to
isolate one group. A compactly supported test function then relates
the derivative norm of this group to the norm of the full sum.
The derivative bounds below control the derivatives of that test function.

\begin{lemma}
\label{gub:d2:native-derivatives}
For $F$ in \eqref{gub:d2:one-group} and every integer $l\ge0$,
\begin{equation}\label{gub:d2:native-derivative-bound}
\begin{aligned}
 \|\sigma_D^{(l)}\|_{L^\infty(\R)}&\le2l!(4/\pi)^l,\\
 \|F^{(l)}\|_{L^\infty(\R)}
 &\le2\sqrt K\,l!(4/\pi)^l\|v\|_2.
 \end{aligned}
\end{equation}
\end{lemma}
\begin{proof}
For $z=x+\ii y$ with $x,y\in\R$,
\[
 |\tanh z|^2=
 \frac{\sinh^2x+\sin^2y}{\sinh^2x+\cos^2y}\le1
 \quad\text{if }|y|\le\pi/4.
\]
The strip contains no poles, and $|\sigma_D(z)|\le2$ there. Cauchy's
formula on circles of radius $\pi/4$ centered on the real axis gives the
first inequality. Summing over the translates and using
$\|v\|_1\le\sqrt K\|v\|_2$ gives the second.
\end{proof}

\begin{theorem}[Tanh-difference coefficient recovery]
\label{gub:d2:global-main}\label{gub:d:general-main}
Assume \eqref{gub:eq:domain-sandwich} and use the direction selectors of
Proposition~\ref{gub:ass:selectors}, with common even degree $r$ and
coefficient bound $B_{\mathrm{sel}}$. For the prescribed tanh-difference
midpoint dictionary, every $c\in\C^{JK}$ satisfies
\begin{equation}\label{gub:d:general-coefficient}
 \|V_c\|_{L^2(\Omega)}\ge e^{-CM^2/T_M}\|c\|_2
\end{equation}
for all sufficiently large real $M$.
More generally, let $K\ge1$ be an integer and let $A,h>0$.
Consider the ridge sum \eqref{gub:d2:detail:physical-network} with any common
width $\eps>0$ and real centered offsets satisfying
$|\beta_{\nu j}|\le A$ for all $\nu,j$, and
$|\beta_{\nu i}-\beta_{\nu j}|\ge h$ for all $\nu$ and $i\ne j$.
Define $E_0$ by \eqref{gub:d2:E0} with these $K,\eps,h$.
If $0<\eps\le1$ and $4\sqrt d\,\eps\le\rho$, then
every $c\in\C^{JK}$ satisfies
\begin{equation}\label{gub:d2:global-parameter-bound}
 \|V_c\|_{L^2(\Omega)}
 \ge\frac{\eps^{d/2}}{B_{\mathrm{sel}}\sqrt{KJ}}
 e^{-C_{A,d}[E_0+(r+1)\log(r+2)]}\|c\|_2.
\end{equation}
Here $C_{A,d}$ depends only on $A,d$. For the midpoint dictionary,
take $A=B_0$, $h=h_M$, and $\eps=\eps_{D,M}$. The constant $C$ and
the threshold for $M$ depend only on the fixed construction and
geometric parameters.
\end{theorem}
\begin{proof}
\emph{Fixed-center scaling and support.}
Write $c_\nu=(c_{\nu j})_{j=1}^K$ and rescale about $x_0$:
\begin{equation}\label{gub:d2:detail:fixed-center-scaling}
\begin{aligned}
 \Omega_\eps&=\{y:x_0+\eps y\in\Omega\},\\
 V_*(y)&=V_c(x_0+\eps y)=\sum_{\nu=1}^JF_\nu(\omega_\nu\cdot y),
 \end{aligned}
\end{equation}
where
$F_\nu(t)=\sum_jc_{\nu j}\sigma_D(t-a_{\nu j})$ and
$a_{\nu j}=\beta_{\nu j}/\eps$.
The scaled offsets satisfy \eqref{gub:d2:scaled-assumptions} uniformly in
$\nu$, since
\begin{equation}\label{gub:d2:detail:scaled-bias-check}
\begin{aligned}
 |a_{\nu j}|&\le A/\eps,\\
 |a_{\nu i}-a_{\nu j}|&=|\beta_{\nu i}-\beta_{\nu j}|/\eps
                         \ge h/\eps\quad(i\ne j).
\end{aligned}
\end{equation}
For each $\nu$, choose an orthogonal matrix $Q_\nu$ with first column
$\omega_\nu$ and write $y=Q_\nu z$, where
$z=(t,u)\in\R\times\R^{d-1}$. Then $t=\omega_\nu\cdot y$.
For $z\in[-2,2]^d$,
\[
 |x_0+\eps Q_\nu z-x_0|=\eps|z|\le2\sqrt d\,\eps\le\rho/2.
\]
The cube $x_0+\eps Q_\nu[-2,2]^d$ therefore lies in
$\overline B_{\rho/2}(x_0)$, at distance at least $\rho/2$ from
$\R^d\setminus\Omega$.

\emph{Selection and the nonnegative testing identity.}
Take the cutoffs of Lemma~\ref{gub:g:cutoff} with the scale parameter
fixed at $a=2$ throughout the proof:
\begin{equation}\label{gub:d2:detail:scaled-cutoffs}
\begin{aligned}
 \eta_r(t)&=\theta_r(t/2),\\
 \chi_r(u)&=\prod_{l=1}^{d-1}\frac{\theta_r(u_l/2)}{2Z_r},\\
 W_r(t,u)&=\eta_r(t)\chi_r(u),
 \end{aligned}
\end{equation}
where $\theta_r$ and $Z_r$ are defined in that lemma.
Then $\int\chi_r=1$, $\chi_r\ge0$, and
$\eta_r(t)\ge(3/4)^{r+1}$ for $t\in I$. The cutoff lemma
also gives support in $[-2,2]^d$ and
$\|\partial^\alpha W_r\|_{L^\infty(\R^d)}\le
 e^{C_d(r+1)\log(r+2)}$ for $|\alpha|\le r$.
Define the test function $\psi_\nu:\R^d\to\C$ by setting
\[
 \begin{aligned}
 \mathcal D_\nu&=P_\nu(\nabla_y), &G_\nu&=F_\nu^{(r)},\\
 \psi_\nu(Q_\nu(t,u))&=W_r(t,u)G_\nu(t),
 &U_\nu&=\|G_\nu\|_{L^2(I)}.
 \end{aligned}
\]
The support of $\psi_\nu$ lies in $Q_\nu[-2,2]^d$, which is compactly contained in
$\Omega_\eps$. Since $P_\nu$ has real coefficients and even degree $r$,
the formal $L^2$ adjoint is
$\mathcal D_\nu^*=(-1)^r\mathcal D_\nu=\mathcal D_\nu$.
For every group $\mu$, homogeneity of degree $r$ and the chain rule give
\[
 P_\nu(\nabla_y)F_\mu(\omega_\mu\cdot y)
 =P_\nu(\omega_\mu)F_\mu^{(r)}(\omega_\mu\cdot y).
\]
The selector values give
\begin{equation}\label{gub:d2:detail:exact-selection}
 \mathcal D_\nu V_*=G_\nu(\omega_\nu\cdot y).
\end{equation}
All derivatives here are taken with respect to $y$.
The orthogonal change of variables has Jacobian of absolute value one.
The normalization in the transverse variables and the positivity of the
cutoff yield
\begin{equation}\label{gub:d2:positive-energy}
\begin{aligned}
 \int_{\Omega_\eps}(\mathcal D_\nu V_*)\overline{\psi_\nu}\dd y
 &=\int_\R\eta_r(t)|G_\nu(t)|^2\dd t\\
 &\ge e^{-C(r+1)}U_\nu^2.
 \end{aligned}
\end{equation}

\emph{Derivative bounds for the test function.}
Lemma~\ref{gub:d2:native-derivatives} gives, for all orders $0\le l\le2r$,
\begin{equation}\label{gub:d2:detail:derivatives-to-twice-r}
 \|F_\nu^{(l)}\|_{L^\infty(\R)}
 \le\sqrt K e^{C(r+1)\log(r+2)}\|c_\nu\|_2.
\end{equation}
The inequalities $l!\le(2r)!$ and
$(4/\pi)^l\le\max\{1,4/\pi\}^{2r}$ give a bound uniform over these orders.
For a multi-index $\alpha$ with $|\alpha|\le r$, the weak Leibniz rule
and the fact that $F_\nu^{(r)}$ depends only on $t=z_1$ give
\begin{equation}\label{gub:d2:detail:test-leibniz}
\partial_z^\alpha[W_r(z)F_\nu^{(r)}(t)]
 =\sum_{k=0}^{\alpha_1}\binom{\alpha_1}{k}
     (\partial_z^{\alpha-ke_1}W_r)(z)F_\nu^{(r+k)}(t).
\end{equation}
Here $e_1=(1,0,\ldots,0)$.
The cutoff is differentiated at most $r$ times, and $F_\nu$
at most $2r$ times. The sum of the binomial coefficients is
$2^{\alpha_1}\le2^r$. Since the support volume is $4^d$, we obtain
\begin{equation}
\begin{aligned}
 \|\partial_z^\alpha[\psi_\nu(Q_\nu z)]\|_{L^2(\R^d)}
 &\le4^{d/2}2^r e^{C_d(r+1)\log(r+2)}
             \sqrt K e^{C(r+1)\log(r+2)}\|c_\nu\|_2\\
 &\le\sqrt K e^{C_d'(r+1)\log(r+2)}\|c_\nu\|_2.
\end{aligned}
 \label{gub:d2:detail:test-all-derivatives}
\end{equation}
Under $\nabla_y=Q_\nu\nabla_z$, the rotated symbol is
$P_\nu(Q_\nu\,\cdot)$. Lemma~\ref{gub:lem:rotation} bounds its
coefficient $\ell^1$ norm by $d^{r/2}B_{\mathrm{sel}}$. Combining this
bound with \eqref{gub:d2:detail:test-all-derivatives} and absorbing
$d^{r/2}$ into the exponential gives
\begin{equation}\label{gub:d2:adjoint-test-bound}
 \|\mathcal D_\nu^*\psi_\nu\|_{L^2(\Omega_\eps)}
 \le B_{\mathrm{sel}}\sqrt K
 e^{C_d(r+1)\log(r+2)}\|c_\nu\|_2.
\end{equation}

\emph{Integration by parts and rescaling.}
The weak derivative bounds and compact support give
$\psi_\nu\in W^{r,2}(\R^d)$. Since its support is compactly contained in
$\Omega_\eps$, mollification approximates $\psi_\nu$ in $W^{r,2}$ by
functions in $C_c^\infty(\Omega_\eps)$. Integration by parts against
these approximations, followed by passage to the $W^{r,2}$ limit, gives
\begin{equation}
\begin{aligned}
 \left|\int_{\Omega_\eps}(\mathcal D_\nu V_*)\overline{\psi_\nu}\dd y\right|
 &=\left|\int_{\Omega_\eps}V_*\overline{\mathcal D_\nu^*\psi_\nu}\dd y\right|\\
 &\le\|V_*\|_{L^2(\Omega_\eps)}
          \|\mathcal D_\nu^*\psi_\nu\|_{L^2(\Omega_\eps)}.
\end{aligned}
 \label{gub:d2:detail:integration-by-parts}
\end{equation}
The change of variables $x=x_0+\eps y$ has
$\dd x=\eps^d\dd y$, and hence
\begin{equation}\label{gub:d2:detail:physical-jacobian}
 \|V_*\|_{L^2(\Omega_\eps)}^2
 =\eps^{-d}\|V_c\|_{L^2(\Omega)}^2.
\end{equation}
Combining \eqref{gub:d2:positive-energy},
\eqref{gub:d2:adjoint-test-bound}, the integration-by-parts identity,
and the change of variables gives
\begin{equation}\label{gub:d2:energy-coefficient-pairing}
 e^{-C(r+1)}U_\nu^2
 \le\eps^{-d/2}\|V_c\|_{L^2(\Omega)}B_{\mathrm{sel}}\sqrt K
 e^{C_d(r+1)\log(r+2)}\|c_\nu\|_2.
\end{equation}

\emph{Recovery of a maximal direction group.}
The estimate is immediate for $c=0$. If $c\ne0$, choose a direction
group with maximal coefficient norm, so that
$\|c_\nu\|_2\ge J^{-1/2}\|c\|_2>0$.
Theorem~\ref{gub:d2:local-recovery} gives
\[
 U_\nu\ge e^{-C_A[E_0+r\log(r+2)]}\|c_\nu\|_2.
\]
Square this bound and substitute it into
\eqref{gub:d2:energy-coefficient-pairing} to obtain
\begin{equation}
\begin{aligned}
 &e^{-C(r+1)-2C_A[E_0+r\log(r+2)]}\|c_\nu\|_2^2\\
 &\qquad\le\eps^{-d/2}\|V_c\|_{L^2(\Omega)}B_{\mathrm{sel}}\sqrt K
         e^{C_d(r+1)\log(r+2)}\|c_\nu\|_2.
\end{aligned}
 \label{gub:d2:detail:group-cancellation}
\end{equation}
Canceling the nonzero factor $\|c_\nu\|_2$ and using maximality yields
\eqref{gub:d2:global-parameter-bound}.

\emph{Evaluation for the midpoint dictionary.}
For the prescribed midpoint dictionary, take $A=B_0$ and $h=h_M$.
The bounds $K\le CM$, $\eps=T_M/M$, and $h_M\ge\kappa_0/M$ give
\begin{equation}\label{gub:d2:actual-E0}
 E_0\le CM\left[1+\frac{M}{T_M}+
                 \log\left(2+\frac{T_M}{\kappa_0}\right)\right].
\end{equation}
Since $\eps=T_M/M\to0$, the conditions $\eps\le1$ and
$4\sqrt d\,\eps\le\rho$ hold for all sufficiently large $M$.
Proposition~\ref{gub:ass:selectors} and
\eqref{gub:eq:full-direction-count} give $r\le C_dM$ and
$J\le C_dM^{d-1}$. The terms in the logarithm of
\eqref{gub:d2:global-parameter-bound} can therefore be bounded as follows:
\begin{equation}\label{gub:d2:detail:final-costs}
\begin{aligned}
 K&\le CM,\\
 K\eps^{-1}&\le CM^2/T_M,\\
 K\log(2+\eps/h_M)&\le CM\log(2+T_M/\kappa_0),\\
 (r+1)\log(r+2)&\le CMT_M,\\
 \log B_{\mathrm{sel}}&\le CMT_M,\\
 \log(KJ)&\le C_dT_M,\\
 \log(\eps^{-d/2})&\le C_dT_M.
 \end{aligned}
\end{equation}
These bounds give
\[
 \|V_c\|_{L^2(\Omega)}\ge e^{-C(M^2/T_M+MT_M)}\|c\|_2.
\]
Since $(MT_M)/(M^2/T_M)=T_M^2/M\to0$, all terms can be absorbed into $CM^2/T_M$.
The constants and threshold depend only on the fixed construction and
geometric parameters. They are uniform with respect to the rounding in
$h_M$.
\end{proof}

\section{Gram eigenvalues and condition numbers}
\label{gub:sec:spectrum}

Fix an integer $m\ge0$ and an activation
$\sigma\in\{\sigma_G,\sigma_D\}$. Except in the Gauss-Legendre extension
below, we use the midpoint columns $g_1,\ldots,g_N$ from
Section~\ref{gub:sec:actual-dictionary}, with
$N=N_M^{\mathrm{eq}}$, $\eps=\eps_{\sigma,M}$, and
$d_i=\|g_i\|_{H^m(\Omega)}>0$. Theorems~\ref{gub:g:coefficient-recovery}
and~\ref{gub:d2:global-main} provide a common constant $C_{\mathrm{rec}}>0$
and a common threshold such that, for either activation, every $c\in\C^N$, and
all sufficiently large real $M$,
\begin{equation}\label{gub:eq:unified-recovery}
 \|V_c\|_{L^2(\Omega)}\ge e^{-C_{\mathrm{rec}}M^2/T_M}\|c\|_2.
\end{equation}
The full $H^m$ norm contains the $L^2$ term. The coefficient-recovery
estimate therefore yields a lower bound for the smallest Gram eigenvalue
for every fixed $m$. Upper bounds for the column norms give an upper
bound for the largest eigenvalue and control the effect of column normalization.

\subsection{Recovery exponents}
\label{gub:sec:exponent-budget}
We estimate the contributions to the two recovery exponents. The signed offset
count $K$, midpoint spacing $h_M$, common selector degree $r$, and
selector coefficient bound $B_{\mathrm{sel}}$ satisfy
\[
 \begin{aligned}
 K&\le C M,\qquad &h_M&\ge\kappa_0/M,\\
 r&\le C_d M,\qquad &\log B_{\mathrm{sel}}&\le C_dMT_M.
 \end{aligned}
\]
These bounds are uniform over the direction groups.

For Gaussian features, propagation on a fixed complex neighborhood
contributes $C\eps_{G,M}^{-2}=CM^2/T_M$ to the exponent. The
fixed-band estimate, derivative factors, and selector coefficients contribute
at most
\begin{equation}\label{gub:eq:gaussian-budget}
 C\bigl[KT_M+r\log(r+2)+r\log(1/\eps_{G,M})
         +\log B_{\mathrm{sel}}\bigr]\le C_dMT_M.
\end{equation}
Here $\log(1/\eps_{G,M})\le T_M$ for large $M$. Thus the Gaussian
recovery exponent is $C(M^2/T_M+MT_M)$.

For the tanh difference, the quantity $E_0$ defined in
\eqref{gub:d2:E0}, evaluated at the midpoint spacing, satisfies
\begin{equation}\label{gub:eq:tanh-budget}
 \begin{aligned}
 E_0&=K\bigl[1+\eps_{D,M}^{-1}+\log(2+\eps_{D,M}/h_M)\bigr]\\
 &\le CM\bigl[1+M/T_M+\log(2+T_M/\kappa_0)\bigr]\\
 &\le C\bigl[M^2/T_M+M\log(T_M+2)+M\bigr].
 \end{aligned}
\end{equation}
In \eqref{gub:d2:global-parameter-bound}, the derivative, selector, and
cutoff terms contribute $O_d(MT_M)$, while the group count and
rescaled volume contribute $O_d(T_M)$.
The lower-order terms are absorbed uniformly for sufficiently large
$M$ by the comparisons
\begin{equation}\label{gub:eq:budget-absorption}
\begin{aligned}
 \frac{MT_M}{M^2/T_M}&=\frac{T_M^2}{M}\longrightarrow0,\\
 \frac{M\log(T_M+2)}{M^2/T_M}
 &=\frac{T_M\log(T_M+2)}{M}\longrightarrow0.
 \end{aligned}
\end{equation}
These comparisons give a common recovery threshold for both activations,
all direction groups, and all coefficient vectors.

\subsection{Column norms and Gram eigenvalues}

\begin{lemma}\label{gub:lem:column-norm}
For $\Omega\subset B_R(0)$, an integer $m\ge0$, a unit vector $\omega$,
$b\in\R$, and $0<\eps\le1$, either activation satisfies
\begin{equation}\label{gub:eq:column-norm}
 \left\|\sigma\!\left(\frac{\omega\cdot x-b}{\eps}\right)
 \right\|_{H^m(\Omega)}^2
 \le C_{d,m,R,\sigma}\eps^{1-2m}.
\end{equation}
The constant is independent of the offset and direction.
\end{lemma}
\begin{proof}
For $|\alpha|=l\le m$,
\[
 \partial^\alpha\sigma((\omega\cdot x-b)/\eps)
 =\eps^{-l}\omega^\alpha\sigma^{(l)}((\omega\cdot x-b)/\eps),
 \qquad |\omega^\alpha|\le1.
\]
In the coordinates $x=t\omega+y$, $y\perp\omega$, every transverse section
of $\Omega\subset B_R(0)$ has $(d-1)$-dimensional measure at most
$C_{d,R}$. Fubini's theorem and the substitution $u=(t-b)/\eps$ give
\begin{align*}
 \|\partial^\alpha\sigma((\omega\cdot x-b)/\eps)\|_{L^2(\Omega)}^2
 &\le C_{d,R}\eps^{-2l}\int_\R
 |\sigma^{(l)}((t-b)/\eps)|^2\dd t\\
 &=C_{d,R}\eps^{1-2l}\|\sigma^{(l)}\|_{L^2(\R)}^2.
\end{align*}
Each Gaussian derivative is a polynomial multiplied by a Gaussian.
For the tanh-difference activation, the identity
\[
 \sigma_D(z)=\frac{2\sinh2}{\cosh(2z)+\cosh2}
\]
gives $|\sigma_D(z)|\le Ce^{-2|\operatorname{Re}z|}$ on the fixed
strip $|\operatorname{Im}z|\le\pi/8$. Indeed, the real part of the
denominator is at least
$2^{-1/2}\cosh(2\operatorname{Re}z)+\cosh2$.
On the circle $|z-t|=\pi/16$ centered at a real $t$,
$|\operatorname{Re}z|\ge|t|-\pi/16$, so Cauchy's formula gives
$|\sigma_D^{(l)}(t)|\le C_le^{-2|t|}$.
Thus the required scalar derivatives of both activations belong to $L^2(\R)$. Summing
over $|\alpha|\le m$ and using $\eps^{1-2l}\le\eps^{1-2m}$ proves
\eqref{gub:eq:column-norm}.
\end{proof}

\begin{proof}[Proof of Theorem~\ref{gub:thm:main}]
\emph{The smallest raw eigenvalue.}
Because the full Sobolev norm contains the zero-order term,
\[
 \begin{aligned}
 c^*H_{m,\sigma,M}c&=\|V_c\|_{H^m(\Omega)}^2\\
 &\ge\|V_c\|_{L^2(\Omega)}^2
 \ge e^{-2C_{\mathrm{rec}}M^2/T_M}\|c\|_2^2.
 \end{aligned}
\]
Thus
\begin{equation}\label{gub:eq:lambda-min}
 \lambda_{\min}(H_{m,\sigma,M})\ge e^{-2C_{\mathrm{rec}}M^2/T_M}>0.
\end{equation}
This bound proves that $H_{m,\sigma,M}$ is positive definite. It also
proves linear independence of the columns, since $V_c=0$ implies $c=0$.

\emph{The largest raw eigenvalue.}
For a positive semidefinite matrix, the largest eigenvalue does not
exceed the trace
\cite[Sections~4.2 and~4.3]{HornJohnson2013}. Applying
Lemma~\ref{gub:lem:column-norm} to its diagonal entries gives
\begin{equation}\label{gub:eq:lambda-max}
 \lambda_{\max}(H_{m,\sigma,M})\le\tr H_{m,\sigma,M}
 \le C_{d,m,R,\sigma}N\eps^{1-2m}.
\end{equation}
Hence
\begin{equation}\label{gub:eq:raw-final}
 \begin{aligned}
 \kappa_2(H_{m,\sigma,M})&\le CN\eps^{1-2m}e^{2C_{\mathrm{rec}}M^2/T_M}\\
 &\le e^{C_mM^2/T_M}.
 \end{aligned}
\end{equation}
For the last inequality, the logarithm of the prefactor satisfies
\[
 \log(CN\eps^{1-2m})
 \le |\log C|+\log N+|1-2m|\,|\log\eps|\le CT_M.
\]
Here $N\asymp M^d$, and both prescribed scales satisfy
$|\log\eps|=O(T_M)$. Since $T_M=o(M^2/T_M)$, the prefactor is
absorbed by increasing the constant in \eqref{gub:eq:raw-final}.

\emph{The normalized eigenvalues.}
Set $d_{\max}=\max_i d_i$. For $z\in\C^N$, substitute $c_i=z_i/d_i$ into
\eqref{gub:eq:unified-recovery}. This gives
\[
 \begin{aligned}
 z^*\widetilde H_{m,\sigma,M}z
 &\ge e^{-2C_{\mathrm{rec}}M^2/T_M}\sum_i\frac{|z_i|^2}{d_i^2}\\
 &\ge\frac{e^{-2C_{\mathrm{rec}}M^2/T_M}}{d_{\max}^2}\|z\|_2^2.
 \end{aligned}
\]
As $d_{\max}^2\le C\eps^{1-2m}$,
$\lambda_{\min}(\widetilde H_{m,\sigma,M})\ge e^{-C_mM^2/T_M}$ after adjusting
constants. Every normalized diagonal entry equals one, so
$\lambda_{\max}(\widetilde H_{m,\sigma,M})\le\tr\widetilde H_{m,\sigma,M}=N$.
Combining the two bounds gives
\[
 \begin{aligned}
 \kappa_2(\widetilde H_{m,\sigma,M})
 &\le N d_{\max}^2 e^{2C_{\mathrm{rec}}M^2/T_M}\\
 &\le CN\eps^{1-2m}e^{2C_{\mathrm{rec}}M^2/T_M}.
 \end{aligned}
\]
Consequently,
\begin{equation}\label{gub:eq:normalized-final}
 \kappa_2(\widetilde H_{m,\sigma,M})\le e^{C_m'M^2/T_M}.
\end{equation}
This normalization estimate uses only the upper bound on $d_{\max}$.

\emph{The exponent in terms of the column count.}
Since $N_M^{\mathrm{eq}}\asymp M^d$ for each activation,
\[
 \frac{M^2}{\log(M+2)}
 \asymp\frac{(N_M^{\mathrm{eq}})^{2/d}}
 {\log(N_M^{\mathrm{eq}}+2)}.
\]
This comparison proves \eqref{gub:eq:main-N}. We choose common constants
$C_m,M_0$ for both activations, both Gram matrices, and both forms of
the exponent. These constants depend only on $m$
and the fixed construction and geometric parameters.
\end{proof}

\subsection{Gauss-Legendre condition-number upper bounds}
\label{gub:sec:gl-extension}

\begin{corollary}
\label{gub:cor:gl-upper}
Replace the midpoint offsets by the positive Gauss-Legendre nodes
\eqref{gub:eq:gl-biases}, keeping the other fixed parameters of
Theorem~\ref{gub:thm:main}. For each fixed integer $m\ge0$, there are
constants $C_m,C_m',M_0>0$ such that, for either activation and every
real $M\ge M_0$, the raw and column-normalized full $H^m$ Gram
matrices are positive definite and satisfy
\begin{equation}\label{gub:eq:gl-upper}
\begin{aligned}
 \max\{\kappa_2(H_{m,\sigma,M}^{\mathrm{GL}}),
       \kappa_2(\widetilde H_{m,\sigma,M}^{\mathrm{GL}})\}
 &\le\exp\!\left(C_m\frac{M^2}{\log(M+2)}\right)\\
 &\le\exp\!\left(C_m'
 \frac{(N_M^{\mathrm{GL}})^{2/d}}{\log(N_M^{\mathrm{GL}}+2)}\right).
 \end{aligned}
\end{equation}
The constants depend only on $m$ and the fixed construction and
geometric parameters.
\end{corollary}
\begin{proof}
The direction selectors are unchanged. In this proof, $g_i,d_i,N$ denote
the GL columns, their $H^m$ norms, and their total number
$N_M^{\mathrm{GL}}$. Set $K=K^{\mathrm{GL}}$ and let $h^{\mathrm{GL}}$
be the minimum distance between distinct signed GL offsets.
Section~\ref{gub:sec:gl-offsets} gives $K\le CM$ and centered offsets
bounded by $B_0$. Its endpoint and spacing bounds
\eqref{gub:eq:gl-separation} give $h^{\mathrm{GL}}\ge C^{-1}M^{-2}$.

In the proof of Lemma~\ref{gub:g:fixed-band}, the new offset separation
gives Vandermonde node separation at least $C^{-1}/(M^2K)$.
The inverse bound used there remains at most $\exp(CKT_M)$, since
$\log(M^2K)=O(T_M)$. The same proof therefore yields the fixed-band
lower bound with exponent $CKT_M$.
Keep $\eps=\eps_{G,M}$, $r\le C_dM$, and the fixed interior observation
interval. The whole-line estimate and tail bound in the proof of
Lemma~\ref{gub:g:fixed-large-interval} remain valid. The tail factor is
$e^{CMT_M-M^2/T_M}$ and can be absorbed using
\eqref{gub:eq:budget-absorption}. The proofs of local propagation and
compactly supported testing give the same recovery exponent
$C(M^2/T_M+MT_M)$, after increasing fixed constants and thresholds.

For the tanh difference, apply \eqref{gub:d2:global-parameter-bound}
with $A=B_0$, $h=h^{\mathrm{GL}}$, and $\eps=\eps_{D,M}$.
The width conditions hold for all sufficiently large $M$. The GL gap gives
\begin{align*}
 E_0&=K\bigl[1+\eps_{D,M}^{-1}+\log(2+\eps_{D,M}/h^{\mathrm{GL}})\bigr]\\
 &\le C\bigl[M^2/T_M+M\log(2+MT_M)+M\bigr]\\
 &\le C(M^2/T_M+MT_M).
\end{align*}
The remaining selector, derivative, and observation factors contribute
$O_d(MT_M)$ to the exponent. Since $MT_M=o(M^2/T_M)$, both activations
satisfy \eqref{gub:eq:unified-recovery} with their GL columns.
The column-norm bound is uniform in the offsets. The preceding spectral
estimates therefore apply to both the raw and normalized GL matrices.
Finally, $N_M^{\mathrm{GL}}\asymp M^d$ gives the stated feature-count form.
\end{proof}

The matching lower bounds in the next section concern midpoint feature families.

\section{Condition-number lower bounds}
\label{gub:sh:section}

We prove lower bounds of the same exponential order as the upper
bound in Theorem~\ref{gub:thm:main} for the raw Gaussian Gram matrix
and for both the raw and column-normalized tanh-difference Gram matrices.
For the normalized Gaussian matrix, we prove two lower bounds of
smaller order: a polynomial bound on every fixed domain and an $\exp(cM)$ lower
bound on a fixed ellipsoid. The matching examples use the rounded midpoint grid
\eqref{gub:eq:half-grid} on the fixed domains
\[
 \Omega_a=\left\{x\in\R^d:
       \frac{x_1^2}{a^2}+\sum_{j=2}^d x_j^2<1\right\},
 \qquad 0<a<1.
\]
They satisfy $B_a(0)\subset\Omega_a\subset B_1(0)$ and
$\max_{\overline\Omega_a}|x|=1$, so the radius prescribed in
Section~\ref{gub:sec:actual-dictionary} is exactly $R=1$.
Their projection onto the first coordinate has range $(-a,a)$.
Offsets in a fixed subinterval of $(a,1)$ therefore lie outside this
projection. The radial offset range is determined by the actual
domain radius.

For the ellipsoidal examples, fix $c_\theta>0$ and take
\[
 q_M=4\lceil c_\theta M\rceil,\qquad
 k_M^{\mathrm{lat}}=2\lceil c_\theta M\rceil+1.
\]
Here $k_M^{\mathrm{lat}}$ is the latitude count determined by $q_M$
in \eqref{gub:eq:degree-and-step-constant}. The circle grid contains
$(1,0)$. Since $k_M^{\mathrm{lat}}$ is odd, the equal-parameter
Jacobi parity identity in Lemma~\ref{gub:hd:antipodal-completion}
gives a zero latitude at every lift. Choosing that latitude at each
lift yields $e_1=(1,0,\ldots,0)\in X_M$ for every $M>0$.
The ellipsoidal examples use this direction and angular sequence.
The polynomial Gaussian lower bound allows any angular sequence in
Theorem~\ref{gub:thm:main}. Throughout this section, $m\ge0$ is a
fixed integer and the domain is fixed as $M$ increases.

To obtain lower bounds for the full feature family, we select columns
from one signed direction group and extend a test coefficient vector
by zero at the remaining indices. For a subset $I$ of column
indices, let $H_I$ be the corresponding principal submatrix of a positive
definite Gram matrix $H$. Cauchy interlacing \cite{HornJohnson2013} gives
\begin{equation}\label{gub:sh:principal-comparison}
\begin{aligned}
 \lambda_{\min}(H)&\le\lambda_{\min}(H_I),\\
 \lambda_{\max}(H)&\ge\lambda_{\max}(H_I),\\
 \kappa_2(H)&\ge\kappa_2(H_I).
 \end{aligned}
\end{equation}
We normalize columns using their full $H^m$ norms on the same domain
$\Omega$. If $D$ is the diagonal matrix of these norms and $D_I$
is its principal submatrix indexed by $I$, then
\[
 (D^{-1}HD^{-1})_I=D_I^{-1}H_ID_I^{-1}.
\]
Thus \eqref{gub:sh:principal-comparison} also applies after column
normalization. The raw Gaussian bound compares the norms of an
interior column and an exterior column. The normalized Gaussian bounds
use normalized columns to form linear combinations with small norms. The
tanh-difference bounds follow by canceling successive terms in the
exponential expansions of exterior columns.

\subsection{Gaussian lower bounds}

In this subsection, let $\eps=\eps_{G,M}=\sqrt{T_M}/M$ and
$h=h_M$ for the Gaussian midpoint family.

\begin{proposition}[Raw Gaussian lower bound]
\label{gub:sh:gaussian-raw-lower}
Let $d\ge2$ be fixed and set $\Omega=\Omega_{1/4}$, so $R=1$. Use
the Gaussian feature family of \eqref{gub:eq:actual-columns}, with
the angular degrees specified at the start of this section.
For every fixed integer $m\ge0$, there exist $c_m>0$ and $M_0$,
independent of $M$, such that
\begin{equation}
 \kappa_2(H_{m,\sigma_G,M})
 \ge \exp\!\left(c_m\frac{M^2}{T_M}\right),\qquad M\ge M_0.
 \label{gub:sh:gaussian-lower}
\end{equation}
\end{proposition}

\begin{proof}
Use the prescribed direction $\omega=e_1$.
Lemma~\ref{gub:lem:rounded-grid} gives $h_M\asymp M^{-1}$.
The midpoint nearest $7/8$ belongs to the positive offset grid for large $M$,
since $B_M\ge1$ and $h_M\to0$. Its distance from $7/8$ is at most
$h_M/2$, so it belongs to $[3/4,1]$ once $h_M\le1/4$.
Fix such an offset $b_M$. For every $x\in\Omega$,
\[
 \frac12\le |\omega\cdot x-b_M|\le\frac54.
\]
The corresponding column is
$g_{\mathrm{far}}(x)=\exp(- (\omega\cdot x-b_M)^2/\eps^2)$.
Let $p_k$ denote the polynomial of degree $k$ defined by
$(\mathrm d/\mathrm dt)^k e^{-t^2}=p_k(t)e^{-t^2}$.
These polynomials satisfy $p_0=1$ and
$p_{k+1}=p_k'-2tp_k$, whence, for each fixed $k$,
$|p_k(t)|\le C_k(1+|t|)^k$.
For $|\alpha|=k\le m$, the chain rule gives
\[
 \partial^\alpha g_{\mathrm{far}}(x)
 =\omega^\alpha\eps^{-k}
 p_k\!\left(\frac{\omega\cdot x-b_M}{\eps}\right)
 \exp\!\left(-\frac{(\omega\cdot x-b_M)^2}{\eps^2}\right).
\]
The bounds on $|\omega\cdot x-b_M|$ and $\eps\le1$ give
\[
 |\partial^\alpha g_{\mathrm{far}}(x)|
 \le C_k\eps^{-2k}e^{-1/(4\eps^2)}.
\]
There are $\binom{d+m}{m}$ multi-indices of order at most $m$.
For fixed $m$, we square and integrate these estimates, then sum over
the multi-indices to obtain
\begin{equation}
 \|g_{\mathrm{far}}\|_{H^m(\Omega)}^2
 \le C_{d,m}|\Omega|\eps^{-4m}
       \exp\!\left(-\frac1{2\eps^2}\right).
 \label{gub:sh:gaussian-Hm-small-diagonal}
\end{equation}

For the first offset $b_{0,M}=h_M/2$, write $x=t\omega+y$ with
$y\perp\omega$. The cylinder
\[
 |t-b_{0,M}|\le\eps/2,\qquad |y|\le1/8
\]
is contained in $B_{1/4}(0)\subset\Omega_{1/4}$ once
$b_{0,M}+\eps/2\le1/8$.
Indeed, $|t|\le1/8$ on this cylinder, and hence
\[
 |x|\le\sqrt{(1/8)^2+(1/8)^2}<1/4.
\]
On this cylinder the square of the near-ridge column
$g_{\mathrm{near}}(x)=\exp(- (\omega\cdot x-b_{0,M})^2/\eps^2)$
is at least $e^{-1/2}$. Its volume is
$\eps |B_{1/8}^{d-1}|$, and hence
\begin{equation}
 \|g_{\mathrm{near}}\|_{H^m(\Omega)}^2
 \ge \|g_{\mathrm{near}}\|_{L^2(\Omega)}^2
 \ge c_d\eps.
 \label{gub:sh:gaussian-large-diagonal}
\end{equation}
Positive definiteness follows from Theorem~\ref{gub:thm:main}.
The coordinate Rayleigh quotient for the far column gives an upper
bound for $\lambda_{\min}$. The quotient for the near column gives a
lower bound for $\lambda_{\max}$. Dividing the near-column lower bound
by the far-column upper bound yields
\begin{equation}
 \kappa_2(H_{m,\sigma_G,M})
 \ge c_{d,m}\eps^{4m+1}
             \exp\!\left(\frac1{2\eps^2}\right).
 \label{gub:sh:gaussian-Hm-diagonal-ratio}
\end{equation}
Since $\eps^{-2}=M^2/T_M$ and
$|\log\eps|=O(T_M)=o(M^2/T_M)$,
\[
 \begin{aligned}
 \log\kappa_2(H_{m,\sigma_G,M})
 &\ge\tfrac12\eps^{-2}-(4m+1)|\log\eps|-C_{d,m}\\
 &\ge\tfrac14\eps^{-2}
 \end{aligned}
\]
after increasing $M_0$ according to $d$ and $m$.
\end{proof}

Estimate \eqref{gub:sh:gaussian-Hm-small-diagonal} bounds the small
diagonal entry in the full $H^m$ norm. Column normalization sets every
diagonal entry to one and removes the norm disparity used in this
proof. The next proposition obtains small eigenvalues from linear
combinations of normalized columns.

\begin{proposition}[Normalized Gaussian lower bounds]
\label{gub:sh:gaussian-normalized-lower}
Consider the Gaussian midpoint feature family of
Theorem~\ref{gub:thm:main}.
For every fixed integer $m\ge0$ and every fixed domain in that theorem,
there is $M_0$ such that
\begin{equation}\label{gub:sh:gaussian-normalized-polynomial}
 \kappa_2(\widetilde H_{m,\sigma_G,M})
 \ge (M+2)^\gamma,\qquad
 \gamma=\frac{1}{16\kappa_0^2},\qquad M\ge M_0.
\end{equation}
For $\Omega=\Omega_{1/4}$ and the angular degrees specified at the
start of this section, there are $c_m>0$ and $M_0$ such that
\begin{equation}\label{gub:sh:gaussian-normalized-exterior}
 \kappa_2(\widetilde H_{m,\sigma_G,M})\ge e^{c_mM},
 \qquad M\ge M_0.
\end{equation}
The thresholds and $c_m$ may depend on the fixed domain, $m$, and
construction parameters.
\end{proposition}

\begin{proof}
\emph{Interior translates on a fixed domain.}
Choose the interior ball in \eqref{gub:eq:domain-sandwich} and any
representative direction $\omega$ from the prescribed direction set.
Set $\tau=\eps/h$ and $n=\lceil\tau^2\rceil$.
By Lemma~\ref{gub:lem:rounded-grid}, $\tau\asymp\sqrt{T_M}$ and
$nh=O(T_M/M)\to0$.
Relabel a block of $n+1$ consecutive signed midpoint offsets as
$\widehat b_j=\widehat b_0+jh$, $0\le j\le n$, and set
\[
 \beta_j:=\widehat b_j-\omega\cdot x_0.
\]
The block can be chosen so that $|\beta_j|\le\rho/4$ for every $j$.
Indeed, $|\omega\cdot x_0|+\rho\le R$, so a fixed neighborhood
of $\omega\cdot x_0$ lies strictly inside the signed offset range.
The grid spacing is $h$, and the selected block has length $nh\to0$.
These signed offsets index actual columns by
\eqref{gub:eq:signed-grid}.

Define
\[
 \begin{aligned}
 f_j(t)&=\sigma_G((t-\beta_j)/\eps),\qquad
 &v_j&=\frac{(-1)^j\binom nj}{\sqrt{\binom{2n}{n}}},\\
 F(t)&=\sum_{j=0}^n v_jf_j(t).
 \end{aligned}
\]
The binomial identity gives $\sum_j|v_j|^2=1$.
Write the columns in the selected block as
$g_j(x)=f_j(\omega\cdot(x-x_0))$.
Choose a coordinate $i$ with $|\omega_i|\ge d^{-1/2}$.
For large $M$, the cylinder
\[
 x=x_0+t\omega+y,\qquad y\perp\omega,\qquad
 |t-\beta_j|\le\eps,\qquad |y|\le\rho/2
\]
lies inside $B_\rho(x_0)$. Integrating
$|\partial_i^m g_j|^2$ on this cylinder gives, including when $m=0$,
\begin{equation}\label{gub:sh:gaussian-interior-column-norm}
 d_j^2:=\|g_j\|_{H^m(\Omega)}^2
 \ge c_{\Omega,d,m}\eps^{1-2m},
\end{equation}
because $\int_{-1}^1|\sigma_G^{(m)}(u)|^2\dd u>0$.

With the convention
$\widehat f(\xi)=\int_\R f(t)e^{-\ii t\xi}\dd t$,
the Gaussian transform and Plancherel's identity give, for $0\le k\le m$,
\[
 \|F^{(k)}\|_{L^2(\R)}^2
 =\frac{\eps^{1-2k}}2\frac{4^n}{\binom{2n}{n}}
 \int_\R u^{2k}e^{-u^2/2}
      \sin^{2n}\!\left(\frac{u}{2\tau}\right)\dd u.
\]
The central binomial coefficient satisfies
$\binom{2n}{n}\ge4^n/(2n+1)$.
On $|u|\le\tau$, the sine factor is at most $4^{-n}$.
On $|u|>\tau$, it is at most one. The Gaussian tail bound
\[
 \int_{|u|>\tau}|u|^{2k}e^{-u^2/2}\dd u
 \le C_k\tau^{2k}e^{-\tau^2/2},\qquad \tau\ge1,
\]
together with $n\ge\tau^2$ and $2n+1\le5\tau^2$, therefore yields
\[
 \|F^{(k)}\|_{L^2(\R)}^2
 \le C_m\eps^{1-2k}\tau^{2m+2}e^{-\tau^2/2}.
\]
Each slice of $\Omega\subset B_R(0)$ perpendicular to $\omega$
has measure at most $C_{d,R}$. Also,
$\partial^\alpha F(\omega\cdot(x-x_0))
 =\omega^\alpha F^{(|\alpha|)}(\omega\cdot(x-x_0))$ and
$\sum_{|\alpha|=k}|\omega^\alpha|^2\le1$.
Fubini's theorem and $\eps\le1$ thus give
\[
 \left\|\sum_{j=0}^n v_jg_j\right\|_{H^m(\Omega)}^2
 \le C\eps^{1-2m}\tau^{2m+2}e^{-\tau^2/2}.
\]
For the normalized columns, take $z_j=v_jd_j$ and extend $z$ by
zero to the full family. By \eqref{gub:sh:gaussian-interior-column-norm},
$\|z\|_2^2\ge c_{\Omega,d,m}\eps^{1-2m}$, while
$\sum_jz_jg_j/d_j=\sum_jv_jg_j$.
The corresponding normalized Rayleigh quotient gives
\[
 \begin{aligned}
 \lambda_{\min}(\widetilde H_{m,\sigma_G,M})
 &\le C\tau^{2m+2}e^{-\tau^2/2}\\
 &\le e^{-\tau^2/4}
 \end{aligned}
\]
for large $M$. Since the normalized diagonal entries are one,
$\lambda_{\max}\ge1$. Finally, $h\le2\kappa_0/M$ gives
$\tau^2\ge T_M/(4\kappa_0^2)$, proving
\eqref{gub:sh:gaussian-normalized-polynomial}.

\emph{Exterior columns.}
For the second estimate, set $\Omega=\Omega_{1/4}$, $a=1/4$, and
$\omega=e_1$. Relabel the positive midpoint offsets in
$[9/16,11/16]$ as $\widehat b_1,\ldots,\widehat b_Q$.
Lemma~\ref{gub:lem:rounded-grid} gives $Q\asymp M$.
For $b\in[9/16,11/16]$, write
\[
 g_b(x)=e^{-(x_1-b)^2/\eps^2},\qquad
 f_b(x)=e^{(b-a)^2/\eps^2}g_b(x).
\]
The normalized columns obtained from $f_b$ and $g_b$ are identical.
For $u=a-x_1>0$ on $\Omega_a$ and a complex parameter $\zeta$,
\[
 f_\zeta(x)=\exp\!\left(-\frac{u^2+2(\zeta-a)u}{\eps^2}\right).
\]
The function is entire in $\zeta$. On $|\zeta-5/8|\le1/4$, we have
$\operatorname{Re}\zeta-a\ge1/8$, so $|f_\zeta(x)|\le1$.
Differentiation in $x_1$ gives, for each fixed $0\le k\le m$,
\[
 |\partial_1^k f_\zeta(x)|\le C_m\eps^{-2m},
 \qquad x\in\Omega_a,\quad |\zeta-5/8|\le1/4.
\]
All derivatives involving the other coordinates vanish.
Cauchy's coefficient estimates on the circle of radius $1/4$
bound the Taylor remainder after degree $Q-2$ in the variable
$b-5/8$, uniformly for $b\in[9/16,11/16]$, by
$C_m\eps^{-2m}4^{-(Q-1)}$ for each of these spatial derivatives.

Choose $v\in\C^Q$ with $\|v\|_2=1$ satisfying
\[
 \sum_{j=1}^Qv_j(\widehat b_j-5/8)^\ell=0,
 \qquad 0\le\ell\le Q-2.
\]
There are $Q-1$ homogeneous equations in $Q$ unknowns, so such a
vector exists. The Taylor polynomial terms cancel, and $\sum_j|v_j|\le\sqrt Q$
gives
\[
 \left\|\sum_{j=1}^Qv_jf_{\widehat b_j}\right\|_{H^m(\Omega_a)}^2
 \le C_mQ\eps^{-4m}4^{-2(Q-1)}.
\]
For a lower bound on each column norm, write $x=(a-u,y)$ with
$y\in\R^{d-1}$ and use the cap
\[
 \eps^2\le u\le2\eps^2,
 \qquad |y|\le\eps.
\]
It lies in $\Omega_a$ for small $\eps$, has volume
$c_d\eps^{d+1}$, and satisfies $f_b(x)\ge c>0$ uniformly over
the selected offsets. Hence
\[
 \|f_{\widehat b_j}\|_{H^m(\Omega_a)}^2
 \ge\|f_{\widehat b_j}\|_{L^2(\Omega_a)}^2
 \ge c_d\eps^{d+1}.
\]
Use $w_j=v_j\|f_{\widehat b_j}\|_{H^m(\Omega_a)}$ as the coefficients
of the normalized columns and extend the vector by zero to the full family.
Then $\|w\|_2^2\ge c_d\eps^{d+1}$, so
\[
 \begin{aligned}
 \lambda_{\min}(\widetilde H_{m,\sigma_G,M})
 &\le C_mQ\eps^{-(4m+d+1)}4^{-2(Q-1)}\\
 &\le e^{-c_mM}.
 \end{aligned}
\]
Here $Q\asymp M$ and the logarithm of the prefactor is $O(T_M)$.
Using $\lambda_{\max}\ge1$ again proves
\eqref{gub:sh:gaussian-normalized-exterior}.
\end{proof}

\subsection{Tanh-difference lower bounds}

Fix
\begin{equation}
 0<a<b_*<b^*<1,\qquad \Omega=\Omega_a,\qquad R=1,
 \qquad \delta=b_*-a>0.
 \label{gub:sh:small-ball-geometry}
\end{equation}
The choice $(a,b_*,b^*)=(1/4,1/2,3/4)$ is admissible.
Figure~\ref{gub:sh:geometry} illustrates the gap $\delta$ between the
projection interval of the domain and the selected ridge centers. Both
this gap and the ellipsoid are independent of $M$.
For the tanh-difference family, set
\begin{equation}\label{gub:sh:actual-halfgrid}
 \begin{aligned}
 \eps&=\eps_{D,M}=\frac{T_M}{M},\qquad &B_M&=1+a_D\frac{T_M^2}{M},\\
 K_M^{\mathrm{eq}}&=\left\lfloor\frac{B_MM}{\kappa_0}\right\rfloor,
 \qquad &h_M&=\frac{B_M}{K_M^{\mathrm{eq}}},\\
 b_{\ell,M}&=(\ell+\tfrac12)h_M,\quad 0\le\ell<K_M^{\mathrm{eq}}.
 \end{aligned}
\end{equation}
The constants $a_D$ and $\kappa_0$ are the same fixed constants as
in the construction of Section~\ref{gub:sec:actual-dictionary}.
The spacing estimates below use the rounded value of $h_M$, with
$h_M\asymp M^{-1}$ as proved in Lemma~\ref{gub:lem:rounded-grid}.
\begin{figure}[tbp]
\centering
\includegraphics[width=\linewidth]{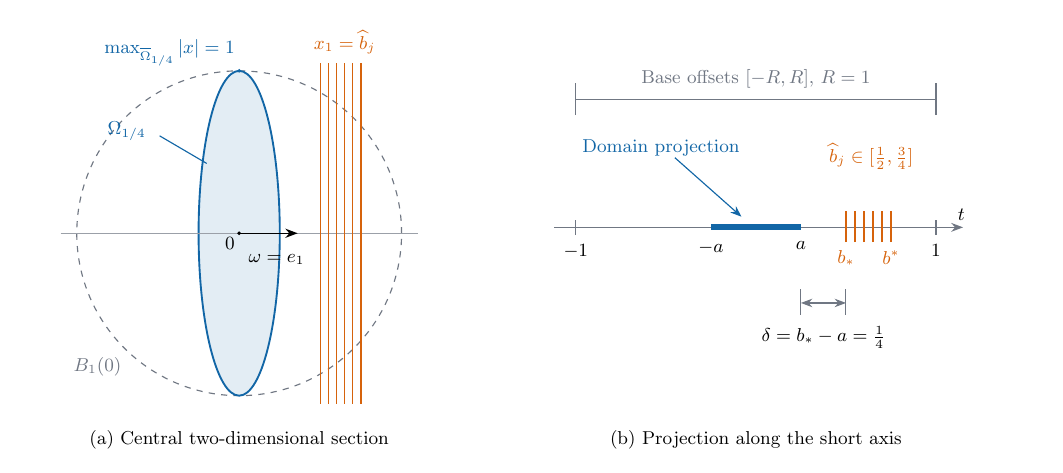}
\caption{Fixed ellipsoid for the tanh-difference lower bound, illustrated
for $a=1/4$, $b_*=1/2$, and $b^*=3/4$. Its outer radius is
$R=1$, while the closure of its projection onto the $e_1$-axis is $[-a,a]$.
The selected offsets lie in $[b_*,b^*]$.
The six lines shown are schematic. The number of selected offsets
in the rounded grid is proportional to $M$. The prescribed envelope $[-B_M,B_M]$
is slightly larger than the base envelope $[-R,R]$, where $R=1$.
The argument applies in every fixed dimension $d\ge2$.}
\label{gub:sh:geometry}
\end{figure}

\begin{theorem}[Raw and normalized tanh-difference lower bounds]
\label{gub:sh:tanh-lower-theorem}
Use the fixed geometry \eqref{gub:sh:small-ball-geometry}, the angular
degrees specified at the start of this section, and the prescribed feature family with
offsets \eqref{gub:sh:actual-halfgrid}. For each
fixed integer $m\ge0$, there are constants $c_m>0$ and $M_0$ such that
\begin{align}
 \kappa_2(H_{m,\sigma_D,M})
 &\ge \exp\!\left(c_m\frac{M^2}{T_M}\right),
 \label{gub:sh:tanh-raw-lower}\\
 \kappa_2(\widetilde H_{m,\sigma_D,M})
 &\ge \exp\!\left(c_m\frac{M^2}{T_M}\right),\qquad M\ge M_0.
 \label{gub:sh:tanh-normalized-lower}
\end{align}
The constants depend only on $d,m$, the fixed geometry, and the fixed
construction parameters.
\end{theorem}

To prove the theorem, use the prescribed direction $\omega=e_1$ and
select all its positive midpoint offsets in $[b_*,b^*]$. Relabel
them as $\widehat b_1,\ldots,\widehat b_Q$ and denote their indices
in the full feature family by $i_1,\ldots,i_Q$.
Since $[b_*,b^*]\subset(0,1)$ and $B_M>1$, the midpoint count at
spacing $h_M$ satisfies
\begin{equation}
 Q=\frac{b^*-b_*}{h_M}+O(1)\asymp M.
 \label{gub:sh:tail-bias-count}
\end{equation}
\begin{lemma}[Tail cancellation]
\label{gub:sh:tail-moment-lemma}
For each sufficiently large $M$, there is a real vector
$v\in\mathbb R^Q$ with $\|v\|_2=1$ such that
\begin{equation}
 \sum_{j=1}^Q v_j e^{-2\ell\widehat b_j/\eps}=0,
 \qquad 1\le\ell<Q.
 \label{gub:sh:vanishing-tail-moments}
\end{equation}
Set
\[
 F_M(t)=\sum_{j=1}^Q v_j
        \sigma_D\!\left(\frac{t-\widehat b_j}{\eps}\right).
\]
For every fixed integer $m\ge0$ and all sufficiently large $M$,
\begin{equation}
 \|F_M^{(k)}\|_{L^\infty([-a,a])}
 \le 4\sqrt Q\left(\frac{2Q}{\eps}\right)^k
       e^{-Q\delta/\eps},\qquad 0\le k\le m.
 \label{gub:sh:tail-derivative-bound}
\end{equation}
\end{lemma}

\begin{proof}
The $Q-1$ homogeneous real equations
\eqref{gub:sh:vanishing-tail-moments} in $Q$ unknowns have a
nonzero solution by the rank-nullity theorem. Normalize this solution in the
Euclidean norm.

For a negative real number $u$,
\[
 \tanh u=-1+\frac{2e^{2u}}{1+e^{2u}}
 =-1+2\sum_{\ell\ge1}(-1)^{\ell-1}e^{2\ell u}.
\]
If $u<-1$, apply this identity to both $u+1$ and $u-1$.
Subtracting the two expansions cancels their constant terms. The identity
$e^{2\ell}-e^{-2\ell}=2\sinh(2\ell)$ gives
\begin{equation}
 \sigma_D(u)=4\sum_{\ell\ge1}(-1)^{\ell-1}
                    \sinh(2\ell)e^{2\ell u}.
 \label{gub:sh:tanh-tail-series}
\end{equation}
The bound $(t-\widehat b_j)/\eps\le-\delta/\eps<-1$ holds
uniformly for $|t|\le a$ and large $M$. Set
\[
 \gamma_M=2(\delta/\eps-1)>0.
\]
After substitution, the terms of the series and its first $m$
derivatives in $t$ are bounded in absolute value by constants times
$\eps^{-k}\ell^k e^{-\gamma_M\ell}$. They therefore converge
absolutely and uniformly on $[-a,a]$.
The majorant justifies termwise differentiation and interchange with
the finite sum over offsets. The moment equations cancel the first
$Q-1$ terms for all derivative orders simultaneously. Differentiating
and summing over the selected offsets gives
\[
 F_M^{(k)}(t)=4\sum_{\ell\ge Q}(-1)^{\ell-1}\sinh(2\ell)
 \left(\frac{2\ell}{\eps}\right)^k
 e^{2\ell t/\eps}
 \sum_{j=1}^Qv_j e^{-2\ell\widehat b_j/\eps}.
\]
For $|t|\le a$, the absolute value of the last two factors is at most
$\sqrt Q\,e^{-2\ell\delta/\eps}$, by the offset lower bound
and the Cauchy-Schwarz inequality.
Using $\sinh(2\ell)\le e^{2\ell}/2$, we obtain
\begin{equation}
 |F_M^{(k)}(t)|
 \le2\sqrt Q\left(\frac2\eps\right)^k
          \sum_{\ell\ge Q}\ell^k e^{-\gamma_M\ell}.
 \label{gub:sh:tail-derivative-series}
\end{equation}
For $0\le k\le m$, the ratio of consecutive summands is at most
$\exp(m/Q-\gamma_M)$, since
$(1+1/\ell)^k\le e^{k/\ell}\le e^{m/Q}$ for $\ell\ge Q$.
Since $Q\to\infty$ and $\gamma_M\to\infty$, this ratio is at most
$1/2$ beyond a common threshold for all $k\le m$. Thus the sum is at
most $2Q^k e^{-\gamma_M Q}$. Finally,
$\gamma_M\ge\delta/\eps$ once $\eps\le\delta/2$, proving
\eqref{gub:sh:tail-derivative-bound}.
\end{proof}

\begin{proof}[Proof of Theorem~\ref{gub:sh:tanh-lower-theorem}]
Extend $v$ by zero to the full feature family. Denote this vector by
$c_M\in\mathbb R^N$, with entries $(c_M)_{i_j}=v_j$ for
$1\le j\le Q$ and zeros elsewhere. It
satisfies $\|c_M\|_2=1$ and
\[
 \sum_i(c_M)_i g_i(x)=F_M(\omega\cdot x).
\]
For a multi-index $\alpha$,
$\partial^\alpha F_M(\omega\cdot x)
=\omega^\alpha F_M^{(|\alpha|)}(\omega\cdot x)$ and
$|\omega^\alpha|\le1$. Since $d,m$ are fixed, summing the
squared bounds of Lemma~\ref{gub:sh:tail-moment-lemma} over the
$\binom{d+m}{m}$ multi-indices yields, for \mbox{$2Q/\eps\ge1$},
\begin{equation}
 c_M^*H_{m,\sigma_D,M}c_M
 \le C_{d,m}|\Omega|Q
       \left(\frac{2Q}{\eps}\right)^{2m}
       \exp\!\left(-\frac{2Q\delta}{\eps}\right).
 \label{gub:sh:Hm-small-rayleigh}
\end{equation}
For the largest eigenvalue of the raw matrix, use the first offset
$b_{0,M}=h_M/2$ in the same direction. Let
$g_{\mathrm{near}}(x)=\sigma_D((\omega\cdot x-b_{0,M})/\eps)$.
In the coordinates $x=t\omega+y$, the cylinder
\[
 |t-b_{0,M}|\le\eps/2,\qquad y\perp\omega,
 \qquad |y|\le a/2
\]
lies in $B_a(0)\subset\Omega_a$ when
$b_{0,M}+\eps/2\le a/2$.
Its volume is $\eps |B_{a/2}^{d-1}|$.
Continuity and positivity of $\sigma_D$ give a positive minimum
on $[-1/2,1/2]$. Hence
\begin{equation}\label{gub:sh:Hm-large-eigenvalue}
\begin{aligned}
 \lambda_{\max}(H_{m,\sigma_D,M})
 &\ge \|g_{\mathrm{near}}\|_{H^m(\Omega)}^2\\
 &\ge \|g_{\mathrm{near}}\|_{L^2(\Omega)}^2\\
 &\ge c_{a,d}\eps.
 \end{aligned}
\end{equation}
Dividing this lower bound by the Rayleigh quotient
\eqref{gub:sh:Hm-small-rayleigh} gives
\[
 \kappa_2(H_{m,\sigma_D,M})
 \ge c_{a,d,m}\frac{\eps}{Q}
       \left(\frac{\eps}{2Q}\right)^{2m}
       e^{2Q\delta/\eps}.
\]
Since $Q\asymp M$ and $\eps=T_M/M$, taking logarithms yields
\[
 \begin{aligned}
 \log\kappa_2(H_{m,\sigma_D,M})
 &\ge\frac{2Q\delta}{\eps}-C_{a,d,m}T_M\\
 &\ge c_m\frac{M^2}{T_M}
 \end{aligned}
\]
for large $M$, proving \eqref{gub:sh:tanh-raw-lower}.

For the normalized matrix, use the column-norm matrix $D_m$ from
\eqref{gub:eq:gram-definition} in the same ordering as $g_i$. Write $d_{i,m}=\|g_i\|_{H^m(\Omega)}$
and set $z_M=D_m c_M$. This change of coefficients gives
\begin{equation}
 z_M^*\widetilde H_{m,\sigma_D,M}z_M=c_M^*H_{m,\sigma_D,M}c_M.
 \label{gub:sh:normalized-energy-identity}
\end{equation}
A lower bound for $\|z_M\|_2$ then gives an upper bound for the
normalized Rayleigh quotient. For $u\le-1$,
\begin{equation}
 \sigma_D(u)
 =\frac{4\sinh(2)e^{2u}}
        {(1+e^{2u+2})(1+e^{2u-2})}
 \ge\sinh(2)e^{2u}.
 \label{gub:sh:tanh-exact-tail-rational}
\end{equation}
The inequality follows because $0<e^{2u+2}\le1$ and
$0<e^{2u-2}\le1$, so the denominator is at most four.
For all selected columns and all $x\in\Omega$,
\[
 -\frac{a+b^*}{\eps}
 \le\frac{\omega\cdot x-\widehat b_j}{\eps}
 \le-\frac\delta\eps\le-1.
\]
It follows that
\begin{equation}\label{gub:sh:far-column-norm-lower}
\begin{aligned}
 d_{i_j,m}^2&\ge\|g_{i_j}\|_{L^2(\Omega)}^2\\
 &\ge|\Omega|\sinh^2(2)
       \exp\!\left(-\frac{4(a+b^*)}{\eps}\right)
 \end{aligned}
\end{equation}
for each selected column, $1\le j\le Q$. Hence
\[
 \|z_M\|_2^2=\sum_{j=1}^Qd_{i_j,m}^2|v_j|^2
 \ge\min_{1\le j\le Q}d_{i_j,m}^2,
\]
because $\|v\|_2=1$. Combining
\eqref{gub:sh:Hm-small-rayleigh},
\eqref{gub:sh:normalized-energy-identity}, and
\eqref{gub:sh:far-column-norm-lower}
in the normalized Rayleigh quotient gives
\begin{equation}
 \lambda_{\min}(\widetilde H_{m,\sigma_D,M})
 \le C_{d,m}Q\left(\frac{2Q}{\eps}\right)^{2m}
 \exp\!\left(-\frac{2Q\delta}{\eps}
                    +\frac{4(a+b^*)}{\eps}\right).
 \label{gub:sh:Hm-normalized-small-eigenvalue}
\end{equation}
The positive term in the exponent is $O(M/T_M)$, whereas
$Q\delta/\eps\asymp M^2/T_M$. To absorb the positive term, choose
$M_0$ so that
$Q\delta\ge4(a+b^*)$. Then
\[
 -\frac{2Q\delta}{\eps}
       +\frac{4(a+b^*)}{\eps}
 \le-\frac{Q\delta}{\eps}.
\]
The logarithm of the remaining prefactor in
\eqref{gub:sh:Hm-normalized-small-eigenvalue} is at most $C_{d,m}T_M$,
whereas $Q\delta/\eps\ge cM^2/T_M$. Increasing $M_0$ again
therefore gives
$\lambda_{\min}(\widetilde H_{m,\sigma_D,M})\le e^{-c_mM^2/T_M}$.
The normalized diagonal entries are one, so a coordinate Rayleigh
quotient gives $\lambda_{\max}(\widetilde H_{m,\sigma_D,M})\ge1$.
The ratio of these eigenvalue bounds proves the normalized estimate.
\end{proof}

The vanishing moments yield the factor $\exp(-cQ/\eps)$ in the
normalized Rayleigh quotient. Division by $\|z_M\|_2^2$ contributes
at most $\exp(C/\eps)$. Since $Q\asymp M$, the cancellation term
determines the exponential order $M^2/T_M$.

\subsection{Optimality over bounded Lipschitz domains}

\begin{corollary}[Matching exponential orders]
\label{gub:sh:matching-corollary}
For the Gaussian feature family on the fixed ellipsoid of
Proposition~\ref{gub:sh:gaussian-raw-lower},
\begin{equation}
 e^{c_mM^2/T_M}\le\kappa_2(H_{m,\sigma_G,M})\le e^{C_mM^2/T_M}.
 \label{gub:sh:gaussian-matching}
\end{equation}
For the tanh-difference feature family on the fixed ellipsoid in
\eqref{gub:sh:small-ball-geometry},
\begin{align}
 e^{c_mM^2/T_M}&\le\kappa_2(H_{m,\sigma_D,M})\le e^{C_mM^2/T_M},
 \label{gub:sh:tanh-raw-matching}\\
 e^{c_mM^2/T_M}&\le\kappa_2(\widetilde H_{m,\sigma_D,M})\le e^{C_mM^2/T_M}.
 \label{gub:sh:tanh-normalized-matching}
\end{align}
These estimates hold for every fixed $d\ge2$ and integer
$m\ge0$, with fixed positive constants and all sufficiently large
$M$, using the angular degrees specified at the start of this section.
In each statement, $M^2/T_M$ can equivalently be replaced by
$N^{2/d}/\log(N+2)$, where $N$ is the actual column count of
the corresponding feature family.
\end{corollary}

\begin{proof}
Combine Proposition~\ref{gub:sh:gaussian-raw-lower} and
Theorem~\ref{gub:sh:tanh-lower-theorem} with
Theorem~\ref{gub:thm:main} on the same domain and the same
feature family. The column count $N\asymp M^d$ implies
\[
 \frac{M^2}{\log(M+2)}
 \asymp\frac{N^{2/d}}{\log(N+2)}.
\]
The counts for the two activations are computed separately, since
their margins and hence their integer offset counts may differ.
\end{proof}

The following table summarizes the matching results for continuous
Gram matrices defined by the full $H^m$ inner product, with $m$ fixed.
\begin{center}
\begin{tabular}{lll}
\hline
Activation & Column scaling & Matching lower-bound geometry\\
\hline
Gaussian & Raw & $\Omega=\Omega_{1/4}$, $R=1$\\
Gaussian & Normalized & Not established (see Proposition~\ref{gub:sh:gaussian-normalized-lower})\\
Tanh difference & Raw & $\Omega=\Omega_a$, $0<a<b_*<b^*<1$\\
Tanh difference & Normalized & The same fixed ellipsoid and offsets\\
\hline
\end{tabular}
\end{center}

For each of the three matrices in the corollary, a fixed ellipsoid
and an admissible angular degree sequence rule out replacing the
common upper exponent by $o(M^2/T_M)$ over the class of fixed bounded
Lipschitz domains, even when the constants depend on the domain and
construction parameters. This optimality concerns the class of domains
and does not prescribe the growth rate on each domain. For the centered ball
$\Omega=B_R(0)$, with $R=\sup_{x\in\Omega}|x|$, every direction
has projection $(-R,R)$, so positive offsets beyond $R$ lie only in the shrinking
margin $(R,R+\Delta_{\sigma,M})$. The fixed exterior intervals used in
the ellipsoidal proofs are therefore unavailable. The present bounds
do not determine the growth rates on the ball. The polynomial bound
of Proposition~\ref{gub:sh:gaussian-normalized-lower} still applies
to normalized Gaussian columns there.

On the short-axis ellipsoid, Gaussian decay makes columns centered
a fixed distance outside the projection interval exponentially smaller
in norm than an interior column. This norm disparity gives the raw
matching lower bound and disappears under column normalization.
Linear combinations of normalized columns give the $\exp(c_mM)$ lower
bound of Proposition~\ref{gub:sh:gaussian-normalized-lower}. A matching
lower bound remains open. The weaker lower bound does not establish
slower growth after normalization.

\section{Numerical observations on a square}
\label{sec:numerics}

We compare the growth of the Gaussian and tanh-difference Gram condition
numbers on $\Omega=(-1,1)^2$ at finite resolutions. Let $G_M$ and
$\widetilde G_M$ denote the raw and column-normalized $L^2(\Omega)$ Gram
matrices, respectively.

The network design uses $s=2$ and the scales and margins in
\eqref{gub:eq:scales} and \eqref{gub:eq:margins}, with $R=\sqrt2$,
$a_G=1$, and $a_D=\pi/2+7/20$. For
$M\in\{3,5,7,9,13\}$, we take $J_M=2M$ equally spaced directions,
including antipodal pairs, and $K_M=M$ positive midpoint offsets in
$(0,B_{\sigma,M})$. Each Gram matrix therefore has $N_M=2M^2$ columns.
The prescribed offset count replaces the rounded count in
\eqref{gub:eq:half-grid} for this numerical comparison.

Figures~\ref{num:fig:gaussian} and~\ref{num:fig:tanh} plot
$\log_{10}\kappa_2(G_M)$ and $\log_{10}\kappa_2(\widetilde G_M)$ against
$M^2/\log(M+2)$. All four condition-number sequences increase across
the sampled resolutions. Column normalization substantially reduces the
Gaussian condition numbers and gives a flatter curve over this range.
For the tanh difference, rapid growth remains after normalization.
The reference lines provide a visual comparison over the plotted range
and do not establish an asymptotic growth law.

\begin{figure}[htbp]
\centering
\includegraphics[width=0.82\linewidth]{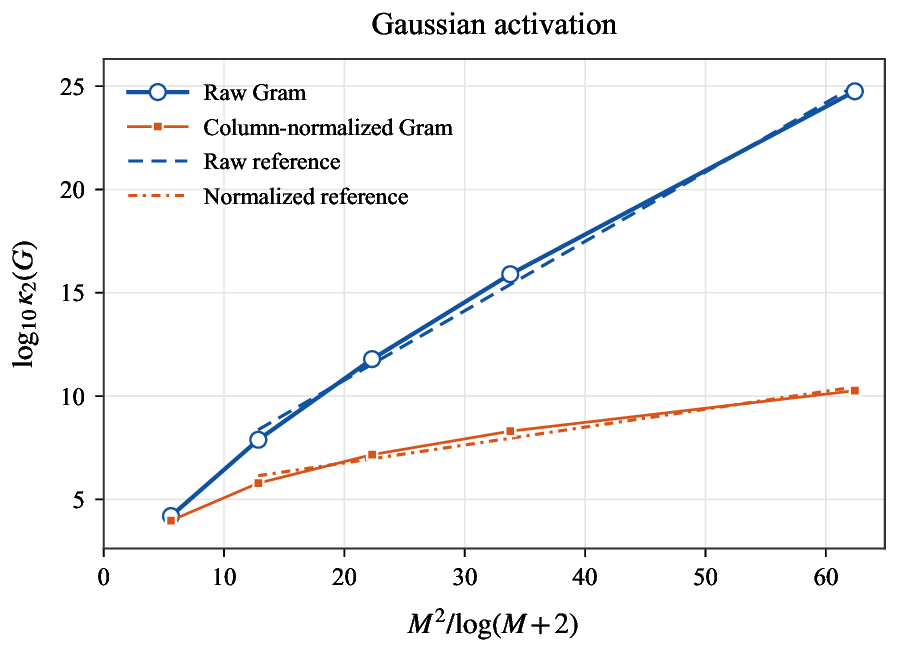}
\caption{Raw and column-normalized Gaussian $L^2$ Gram condition numbers
 on the square, with $2M$ directions and $M$ positive midpoint offsets.}
\label{num:fig:gaussian}
\end{figure}

\begin{figure}[htbp]
\centering
\includegraphics[width=0.82\linewidth]{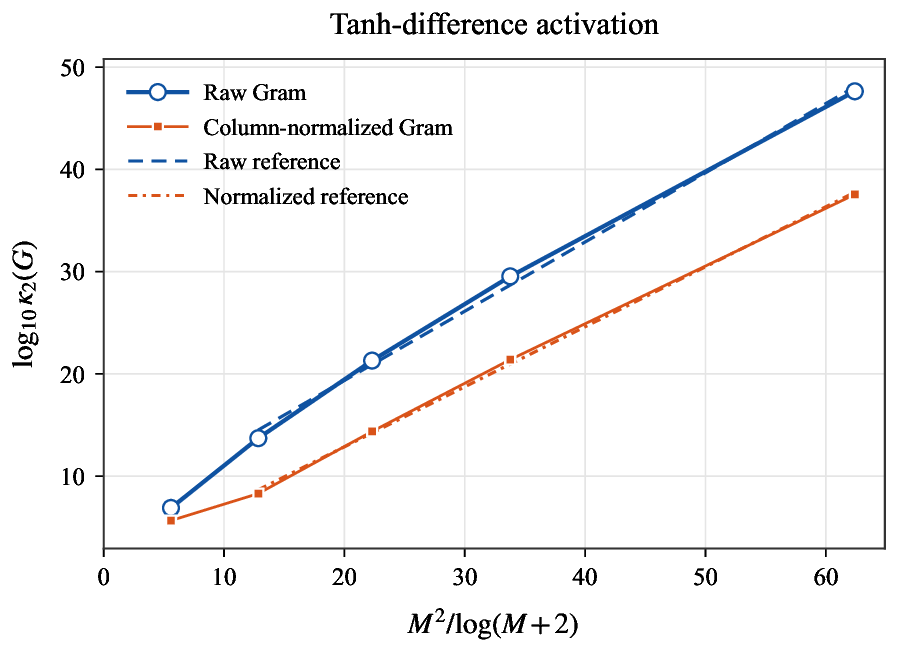}
\caption{Raw and column-normalized tanh-difference $L^2$ Gram condition
 numbers on the same square, with the same node counts as in
 Figure~\ref{num:fig:gaussian}.}
\label{num:fig:tanh}
\end{figure}

\section{Conclusion}

For the prescribed Gaussian and tanh-difference ridge families, we have
proved positive definiteness and the condition-number bound
$\kappa_2\le\exp(C_mM^2/\log(M+2))$ for raw and column-normalized
continuous Gram matrices in the full $H^m$ inner product. The result
applies to midpoint and Gauss-Legendre offsets on every fixed
nonempty bounded Lipschitz domain in dimension $d\ge2$, for fixed
integer $m\ge0$ and sufficiently large real $M$. The actual feature count satisfies
$N\asymp M^d$, so the exponent has order $N^{2/d}/\log(N+2)$.

The proof uses homogeneous direction selectors of degree $O_d(M)$,
one-dimensional coefficient recovery, and observation on an interior
ball. Together, these estimates control arbitrary coefficient vectors
for the prescribed nodes
and remain valid for the rounded midpoint grid. Under the additional
geometric and parameter hypotheses of Part I, they quantify the
coefficient sensitivity of feature families that attain the optimal
Sobolev approximation order.

For midpoint offsets, fixed ellipsoids and an admissible angular
degree sequence yield matching lower bounds for raw Gaussian matrices
and for raw and normalized tanh-difference matrices. These examples
establish optimality over the domain class in the three cases. The
Gaussian proof uses disparities in column norms, whereas the
tanh-difference proof uses cancellation of successive tail terms.
For normalized Gaussian midpoint matrices, we obtain a polynomial lower bound on
every fixed domain and an $\exp(c_mM)$ lower bound on a fixed ellipsoid.
The square-domain $L^2$ experiments illustrate condition-number growth
and the effect of normalization at finite resolutions.

The gap between the lower and upper bounds for normalized Gaussian
matrices remains unresolved. Determining the sharp growth rates on balls
centered at the origin is a related question. With
$R=\sup_{x\in\Omega}|x|$, these balls have projection range $(-R,R)$ in
every direction, which excludes the fixed exterior offset intervals used
in the ellipsoidal proofs. Matching lower bounds for Gauss-Legendre
offsets also remain to be established.

\bibliographystyle{elsarticle-num}
\bibliography{references}
\end{document}